\documentclass[microtype]{gtpart}

\usepackage[utf8]{inputenc}
\usepackage{mathtools}
\usepackage{subcaption}
\usepackage[dvipsnames]{xcolor}
\usepackage[plain]{fancyref}
\usepackage{ulem}
\usepackage{multirow}

\newtheorem{theorem}{Theorem}[section]    
\newtheorem{lemma}[theorem]{Lemma}
\newtheorem{remark}[theorem]{Remark}
\newtheorem{corollary}[theorem]{Corollary}

\newcommand{\dt}{\Delta t}
\newcommand{\dx}{\Delta x}
\newcommand{\real}{\mathbb{R}}
\newcommand{\compl}{\mathbb{C}}
\newcommand{\infnorm}[1]{\lVert#1\rVert_\infty}
\newcommand{\bigo}{\mathcal{O}}
\let\oldexp\exp
\renewcommand{\exp}[1]{\oldexp{\left(#1\right)}}
\DeclarePairedDelimiter\abs{\lvert}{\rvert}%
\DeclarePairedDelimiter\norm{\lVert}{\rVert}%

\newcommand{\fancyrefequlabelprefix}{equ}
\frefformat{plain}{\fancyrefequlabelprefix}{equation\fancyrefdefaultspacing(#1)}
\Frefformat{plain}{\fancyrefequlabelprefix}{Equation\fancyrefdefaultspacing(#1)}

\newcommand{\fancyrefsubseclabelprefix}{subsec}
\frefformat{plain}{\fancyrefsubseclabelprefix}{section\fancyrefdefaultspacing#1}
\Frefformat{plain}{\fancyrefsubseclabelprefix}{Section\fancyrefdefaultspacing#1}

\newcommand{\fancyreftheoremlabelprefix}{theo}
\frefformat{plain}{\fancyreftheoremlabelprefix}{theorem\fancyrefdefaultspacing#1}
\Frefformat{plain}{\fancyreftheoremlabelprefix}{Theorem\fancyrefdefaultspacing#1}

\newcommand{\fancyrefcorollarylabelprefix}{cor}
\frefformat{plain}{\fancyrefcorollarylabelprefix}{corollary\fancyrefdefaultspacing#1}
\Frefformat{plain}{\fancyrefcorollarylabelprefix}{Corollary\fancyrefdefaultspacing#1}

\newcommand{\fancyrefdefinitionlabelprefix}{def}
\frefformat{plain}{\fancyrefdefinitionlabelprefix}{definition\fancyrefdefaultspacing#1}
\Frefformat{plain}{\fancyrefdefinitionlabelprefix}{Definition\fancyrefdefaultspacing#1}

\newcommand{\fancyreflemmalabelprefix}{lem}
\frefformat{plain}{\fancyreflemmalabelprefix}{lemma\fancyrefdefaultspacing#1}
\Frefformat{plain}{\fancyreflemmalabelprefix}{Lemma\fancyrefdefaultspacing#1}

\newcommand{\fancyrefremarklabelprefix}{rem}
\frefformat{plain}{\fancyrefremarklabelprefix}{remark\fancyrefdefaultspacing#1}
\Frefformat{plain}{\fancyrefremarklabelprefix}{Remark\fancyrefdefaultspacing#1}

\newcommand{\fancyrefalgorithmlabelprefix}{alg}
\frefformat{plain}{\fancyrefalgorithmlabelprefix}{algorithm\fancyrefdefaultspacing#1}
\Frefformat{plain}{\fancyrefalgorithmlabelprefix}{Algorithm\fancyrefdefaultspacing#1}

\newcommand{\diff}[1]{#1}

\title{Convergence analysis of multi-level spectral deferred corrections}

\author{Gitte Kremling}
\address{Juelich Supercomputing Centre, Forschungszentrum Juelich GmbH, 52425 Juelich, Germany}
\email{g.kremling@fz-juelich.de}

\author{Robert Speck}
\address{Juelich Supercomputing Centre, Forschungszentrum Juelich GmbH, 52425 Juelich, Germany}
\email{r.speck@fz-juelich.de}


\begin{document}

\begin{abstract}
The spectral deferred correction (SDC) method is class of iterative solvers for ordinary differential equations (ODEs). It can be interpreted as a preconditioned Picard iteration for the collocation problem. The convergence of this method is well-known, for suitable problems it gains one order per iteration up to the order of the quadrature method of the collocation problem provided. This appealing feature enables an easy creation of flexible, high-order accurate methods for ODEs. 
A variation of SDC are multi-level spectral deferred corrections (MLSDC). Here, iterations are performed on a hierarchy of levels and an FAS correction term, as in nonlinear multigrid methods, couples solutions on different levels. While there are several numerical examples which show its capabilities and efficiency, a theoretical convergence proof is still missing. This paper addresses this issue. A proof of the convergence of MLSDC, including the determination of the convergence rate in the time-step size, will be given and the results of the theoretical analysis will be numerically demonstrated. It turns out that there are restrictions for the advantages of this method over SDC regarding the convergence rate.
\end{abstract}

\maketitle

\section{Introduction}

The original spectral deferred correction (SDC) method for solving ordinary differential equations (ODEs), a variant of the defect and deferred correction methods developed in the 1960s \cite{def-corr-1,def-corr-2,def-corr-3,def-corr-4}, was first introduced in \cite{sdc-orig} and then subsequently improved, e.g. in \cite{hansen,huang,imex-1,imex-2}. It relies on a discretization of the initial value problem in terms of a collocation problem which is then iteratively solved using a preconditioned fixed-point iteration. The iterative structure of SDC has been proven to provide many opportunities for algorithmic and mathematical improvements. 
These include the option of using Newton-Krylov schemes such as the Newton-GMRES method to solve the resulting preconditioned nonlinear systems, leading to the so-called Krylov deferred correction methods~\cite{huang,kdc}. Various semi-implicit and multi-implicit formulations of the method have been explored~\cite{hagstrom,imex-1,imex-2,BourliouxEtAl2003,LaytonMinion2004}. In the last decade, SDC has been applied e.g.\ to gas dynamics and incompressible or reactive flows~\cite{appl-1,imex-2} as well as to fast-wave slow-wave problems~\cite{sdc-expl-euler} or particle dynamics~\cite{appl-2}. The generalized integral deferred correction framework includes further variations of SDC, where the used discretization approach is not limited to collocation methods~\cite{idc-1,idc-2}. Moreover, the SDC approach was used to derive efficient parallel-in-time solvers addressing the needs of modern high-performance computing architectures~\cite{pfasst,sdc-notation}.

Here, we will focus on the multi-level extension of SDC, namely multi-level spectral deferred corrections (MLSDC), which was introduced in \cite{mlsdc-1}. It uses a multigrid-like approach to solve the collocation problem with SDC iterations (now called ``sweeps'' in this context) being performed on the individual levels. The solutions on the different levels are then coupled using the Full Approximation Scheme (FAS) coming from nonlinear multigrid methods. This variation was designed to improve the efficiency of the method by shifting some of the work to coarser, less expensive levels.
While there are several numerical examples which show the correctness and efficiency of MLSDC \cite{mlsdc-2,doi:10.1080/13647830.2019.1566574,HAMON2019435}, a theoretical proof of its convergence is still missing. The convergence of SDC, however, was already extensively examined \cite{causley,hagstrom,hansen,huang,xia,tang}. It could be shown that, under certain conditions, the method gains one order per iteration up to the accuracy of the solution of the collocation problem. The aim of this work now is to prove statements on the convergence behavior of MLSDC using similar concepts and ideas as they were used in the convergence proof of SDC, in particular the one presented in \cite{tang}.

For that, we first review SDC along with one of its existing convergence proofs, forming the basis for the following convergence analysis of MLSDC. Then, MLSDC is described and a first convergence theorem is provided. The theorem specifically states that MLSDC behaves at least as good as SDC does. Since this result contradicts our intuitive expectations in that we would assume the multi-level extension to be more efficient than the original one, we will again examine the convergence proof in greater detail, now for a specific choice of transfer operators between the different levels. As a result, a second theorem on the convergence of MLSDC will be derived, describing an improved behavior of the method if particular conditions are fulfilled. 
More specifically, we will provide theoretical guidelines for parameter choices in practical applications of MLSDC in order to achieve this improved efficiency. Finally, the theoretical results will be verified by numerical examples.

\section{Spectral Deferred Corrections}
\label{sec:sdc}

In the following SDC is presented as preconditioned Picard iterations for the collocation problem. The used approach and notations are substantially based on \cite{mlsdc-1,conv-pfasst} and references therein. First, the collocation problem for a generic initial value problem is explained. Then, SDC is described as a solver for this problem and compact notations are introduced. Finally, an existing theorem on the convergence of SDC, including its proof, is presented. 

\subsection{SDC and the collocation problem}

Consider the following autonomous initial value problem (IVP)
\begin{align}\begin{split}
    \label{equ:ivp}
    u'(t) &= f(u(t)), \quad t\in[t_0, t_1],\\
    u(t_0) &= u_0
\end{split}\end{align}
with $u(t), u_0 \in \compl^N$ and $f:\compl^N\to\compl^N$, $N\in\mathbb{N}$. To guarantee the existence and uniqueness of the solution, $f$ is required to be Lipschitz continuous. Since a high-order method shall be used, $f$ is additionally assumed to be sufficiently smooth.

The IVP can be written as
\begin{align*}
    u(t) = u_0 + \int_{t_0}^t f(u(s))ds, \quad t\in[t_0, t_1]
\end{align*}
and choosing $M$ quadrature nodes $\tau_1, ..., \tau_M$ within the time interval such that $t_0 \le \tau_1 < \tau_2 < ... < \tau_M = t_1$, the integral is now approximated using a spectral quadrature rule like Gauß-Radau. 
This approach results in the discretized system of equations
\begin{align}
    \label{equ:intgl1}
    u_m = u_0 + \Delta t \sum_{j=1}^M q_{m,j} f(u_j), \quad m=1,...,M,
\end{align}
where $u_m \approx u(\tau_m)$, $\dt = t_1 - t_0$ denotes the time step size and $q_{m,j}$ represent the quadrature weights for the unit interval with
\begin{align*}
    q_{m,j} = \frac{1}{\dt} \int_{t_0}^{\tau_m} l_j(s) ds.
\end{align*}
Here, $l_j$ represents the $j$-th Lagrange polynomial corresponding to the set of nodes $(\tau_m)_{1 \leq m \leq M}$.
We can combine these $M$ equations into the following system of linear or non-linear equations, defining the collocation problem:
\begin{gather}
	C(U) \coloneqq (I_{MN} - \Delta t (Q \otimes I_N) F)(U) = U_0,
	\label{equ:coll_prob}
\end{gather}
where $U \coloneqq (u_1, u_2, \dots, u_M)^T \in \compl^{MN}$, $U_0 \coloneqq (u_0, u_0, \dots, u_0)^T \in \compl^{MN}$, $Q \coloneqq (q_{m,j})_{1 \leq m,j \leq M}$ is the matrix gathering the quadrature weights, the vector function $F$ is given by $F(U) \coloneqq (f(u_1), f(u_2), \dots, f(u_M))^T$ and $I_{MN}, I_N$ are the identity matrices of dimensions $MN$ and $N$.

As described above, the solution $U$ of the collocation problem approximates the solution of the initial value problem (\ref{equ:ivp}). With this in mind, the following theorem, referring to \cite[Thm.~7.10]{hairer-wanner-1}, provides a statement on its order of accuracy.

\begin{theorem}
	\label{theo:coll_prob_err}
	The solution $U = (u_1, u_2, \dots, u_M)^T \in \compl^{MN}$ of the collocation problem defined by \fref{equ:coll_prob} approximates the solution $u$ of the IVP \eqref{equ:ivp} at the collocation nodes. In particular, for $\bar{U} \coloneqq (u(\tau_1), \dots, u(\tau_M))^T$ the following error estimation applies:
	\begin{align*}
    	\| \bar{U} - U \|_\infty \leq C_1 \Delta t^{M+1} \|u\|_{M+1},
	\end{align*}
	where $C_1$ is independent of $\Delta t$, $M$ denotes the number of nodes and $\|u\|_{M+1}$ represents the maximum norm of $u^{(M+1)}$, the $(M+1)$th derivative of $u$.
\end{theorem}

Interpreting the collocation problem as a discretization method with discretization parameter $n \coloneqq \Delta t^{-1}$, the theorem shows that the discrete approximation $U$ defined by the collocation problem converges with order $M+1$ to the solution $\bar{U}$ of the corresponding IVP.

Since the system of equations \eqref{equ:coll_prob} defining the collocation problem has an unfavorable structure with $Q$ gathering the quadrature weights being fully populated, a direct solution is not advisable, in particular if the right-hand side of the ODE is non-linear. An iterative method to solve the problem is SDC. 

The standard Picard iteration for the collocation problem \eqref{equ:coll_prob} is given by
\begin{align}\begin{split}
    \label{equ:richardson}
    U^{(k+1)} &= U^{(k)} + (U_0 - C(U^{(k)}))\\
    &= U_0 + \Delta t (Q \otimes I_N)F(U^{(k)}).
\end{split}\end{align}

As this method only converges for very small step sizes $\Delta t$, using a preconditioner to increase range and speed of convergence is reasonable. The SDC-type preconditioners are defined by
\begin{align*}
    P(U) = (I_{MN} - \Delta t (Q_\Delta \otimes I_N) F) (U),
\end{align*}
where the matrix $Q_\Delta = (q_{\Delta_{m,j}})_{1 \leq m,j \leq M} \approx Q$ is formed by the use of a simpler quadrature rule. In particular, $Q_\Delta$ is typically a lower triangular matrix, such that solving the system can be easily done by forward substitution.

Common choices for $Q_\Delta$ include the matrix
\begin{align*}
    Q_\Delta = \frac{1}{\Delta t}
    \begin{pmatrix}
        \Delta \tau_1 \\
        \Delta \tau_1 & \Delta \tau_2 \\
        \vdots & \vdots & \ddots \\
        \Delta \tau_1 & \Delta \tau_2 & \hdots & \Delta \tau_M
    \end{pmatrix}
\end{align*}
with $\Delta \tau_m = \tau_m - \tau_{m-1}$ for $m=2,...,M$ and $\Delta \tau_1 = \tau_1 - t_0$ representing the right-sided rectangle rule. Similarly, the left-sided rectangle rule \cite{sdc-expl-euler} or a part of the $LU$ decomposition of the matrix $Q$ \cite{sdc-lu} are chosen. The theoretical considerations in the next chapters do not rely on a specific matrix $Q_\Delta$. However, in the numerical examples the right-sided rectangle rule as given above is used.

By the use of such an operator to precondition the Picard iteration (\ref{equ:richardson}), the following iterative method for solving the collocation problem is obtained
\begin{align}\begin{split}
    \label{equ:sdc}
    (I_{MN} - \Delta t (Q_\Delta \otimes I_N)F) (U^{(k+1)})
    = U_0 + \Delta t ((Q-Q_\Delta) \otimes I_N)F(U^{(k)}),
\end{split}\end{align}
which constitutes the SDC iteration \cite{sdc-orig,huang}. 
Written down line-by-line, this formulation recovers the original SDC notation given in~\cite{sdc-orig}.
A more implicit formulation is given by
\begin{align}\begin{split}
    \label{equ:sdc_implicit}
    U^{(k+1)} = U_0 + \Delta t (Q_\Delta \otimes I_N)F(U^{(k+1)}) + \Delta t ((Q-Q_\Delta) \otimes I_N)F(U^{(k)})
\end{split}\end{align}
and this will be used for the following convergence considerations.

\subsection{Convergence of SDC}

There already exist several approaches proving the convergence of SDC, particularly those presented in \cite{causley,hagstrom,hansen,huang,xia,tang}. Here, we will focus on the idea of the proof from \cite{tang} as it uses the previously introduced matrix formulation of SDC, needed for an appropriate adaptation for a convergence proof of MLSDC, and simultaneously, provides a general result for linear and non-linear initial value problems. 
We will review the idea of this proof in some detail to introduce the notation and the key ideas.
This is followed by a concise discussion on stability and convergence of SDC in the sense of one-step ODE solvers.
The approach in \cite{tang} relies on a split of the local truncation error (LTE). The key concept used in the proof is a property of the operators $QF(U)$ and $Q_\Delta F(U)$, respectively, which can be interpreted as a kind of extended Lipschitz continuity. It is presented in the following lemma using the previously introduced notations. For reasons of readability, the sizes of the identity matrices are no longer denoted here.
        
\begin{lemma}
    \label{lem:estimation}
    If $f:\compl^N \to \compl^N$ is Lipschitz continuous, the following estimates apply
    \begin{align*}\begin{split}
        \infnorm{\Delta t (Q \otimes I) (F(U_1) - F(U_2))}
        &\leq C_2 \Delta t \infnorm{U_1 - U_2},\\
        \infnorm{\Delta t (Q_\Delta \otimes I) (F(U_1) - F(U_2))}
        &\leq C_3 \Delta t \infnorm{U_1 - U_2},
    \end{split}\end{align*}
    where the constants $C_2$ and $C_3$ are dependent on the Lipschitz constant $L$, but independent of $\Delta t$ and $U_1, U_2 \in \compl^{NM}$.
\end{lemma}
\begin{proof}
    
    Can be shown directly using the definition of the maximum norm, the Lipschitz continuity of $f$ and the compatibility between maximum absolute row sum norm for matrices and maximum norm for vectors.
\end{proof}

\begin{remark}
    For a system of ODEs stemming from a discretized PDE, the constants $C_2$ and $C_3$ may depend on the spatial resolution given by some grid spacing $\Delta x$, because the Lipschitz constant of $f$ may depend on it.
    In this case we have $C_2 = C_2(\Delta x^{-d})$, $C_3 = C_3(\Delta x^{-d})$ for $d\in\mathbb{N}$.
    For example, using second-order finite differences in space for the heat equation results in the ODE system $u' = Au$ with matrix $A \in \mathcal{O}(\Delta x^{-2})$, i.e. $d=2$ in this case.
    This has to be kept in mind for most of the upcoming results and we will address this point separately in remarks where appropriate.
    This will be particularly relevant for the convergence results in section~\ref{ssec:improved_mlsdc}, where the spatial discretization plays a key role.
    Note, however, that this is a rather pessimistic estimate. 
    When focusing on spatial operators with more restrictive properties (e.g.~linearity) or when using a specific matrix $Q_\Delta$, the convergence results can be improved substantially, both in terms of constants and time-step size restrictions.
    For SDC, this has been already done, see e.g.~\cite{sdc-lu}.
\end{remark}

The following theorem provides a convergence statement for SDC using the presented lemma in the proof.

\begin{theorem}
   	\label{theo:sdc_conv}
    Consider a generic initial value problem like (\ref{equ:ivp}) with a Lipschitz-continuous function $f$ on the right-hand side. 
    
    If the step size $\Delta t$ is sufficiently small, SDC converges linearly to the solution $U$ of the collocation problem with a convergence rate in $\bigo(\dt)$, i.e. the following estimate for the error of the $k$-th iterated $U^{(k)}$ of SDC compared to the solution of the collocation problem is valid:
   	\begin{align}
   	    \label{equ:sdc_conv_iter}
   		\infnorm{U-U^{(k)}} &\leq C_4 \Delta t \infnorm{U - U^{(k-1)}},
   	\end{align}
   	where the constant $C_4$ is independent of $\Delta t$.
   	
   	If, additionally, the solution of the initial value problem $u$ is $(M+1)$-times continuously differentiable, the LTE of SDC compared to the solution $\bar{U}$ of the ODE can be bounded by
   	\begin{align}\begin{split}
   	    \label{equ:sdc_conv_lte}
        \infnorm{\bar{U} - U^{(k)}} &\leq C_5 \Delta t^{k_0+k} \|u\|_{k_0+1} 
        + C_6 \Delta t^{M+1} \|u\|_{M+1}\\
        &= \bigo(\Delta t^{\min(k_0+k, M+1)}),
    \end{split}\end{align}
    where the constants $C_5$ and $C_6$ are independent of $\Delta t$, $k_0$ denotes the approximation order of the initial guess $U^{(0)}$ and $\|u\|_p$ is defined by $\infnorm{u^{(p)}}$.
\end{theorem}
\begin{proof}
   	We again closely follow \cite{tang} here. According to the definition of the collocation problem (\ref{equ:coll_prob}) and an SDC iteration (\ref{equ:sdc_implicit}), it follows
   	\begin{align*}
   		\infnorm{U-U^{(k)}} &= \infnorm{\Delta t (Q \otimes I) (F(U) - F(U^{(k-1)}))\\
   			&\quad\quad+ \Delta t (Q_\Delta \otimes I) (F(U^{(k-1)}) - F(U^{(k)}))}.
   	\end{align*}
   	Together with the triangle inequality and \fref{lem:estimation}, we obtain
   	\begin{align*}
   		\infnorm{U-U^{(k)}} &\leq C_2 \Delta t \infnorm{U - U^{(k-1)}} + C_3 \Delta t \infnorm{U^{(k-1)} - U^{(k)}}.
   	\end{align*}
   	Applying the triangle inequality again yields
   	\begin{align*}
   		\infnorm{U-U^{(k)}} 
   		&\leq \tilde{C}_1 \Delta t \infnorm{U - U^{(k-1)}} + C_3 \Delta t \infnorm{U - U^{(k)}},
   	\end{align*}
   	where here and in the following, we use variables in the form of $\tilde{C}_i$ to denote temporary arising constants. We continue by subtracting $C_3 \Delta t \infnorm{U - U^{(k)}}$ from both sides and dividing by $1 - C_3 \Delta t$ which results in
    \begin{align*}
        \infnorm{U-U^{(k)}} &\leq \frac{\tilde{C}_1}{1-C_3 \Delta t} \Delta t \infnorm{U - U^{(k-1)}}.
    \end{align*}
   	If the step size is sufficiently small, in particular 
   	\begin{align}\label{eq:conv_cond}
   	    C_3 \Delta t < 1,
   	\end{align} the following estimate is valid
   	\begin{align}\label{eq:conv_cond_II}
   		\frac{\tilde{C}_1}{1-C_3 \Delta t} \leq C_4,
   	\end{align}
    which concludes the proof for \fref{equ:sdc_conv_iter}.
    
    Continuing with recursive insertion, we get
    \begin{align*}
        \infnorm{U-U^{(k)}} \leq \tilde{C}_2 \Delta t^k \infnorm{U - U^{(0)}}.
    \end{align*}
    Since $U^{(0)}$ is assumed to be an approximation of $k_0$-th order, we further know that
    \begin{align*}
        \norm{\bar{U} - U^{(0)}} \leq \tilde{C}_3 \Delta t^{k_0} \norm{u}_{k_0+1}.
    \end{align*}
    This estimation together with the triangle inequality and the error estimation for the solution of the collocation problem stated in \fref{theo:coll_prob_err} yields
    \begin{align}
        \infnorm{U-U^{(k)}} &\leq \tilde{C}_2 \Delta t^k (\infnorm{\bar{U} - U} + \infnorm{\bar{U} - U^{(0)}}) \nonumber\\
        \label{equ:proof_sdc_err_coll}
        &\leq \tilde{C}_4 \Delta t^{M+k+1} \|u\|_{M+1} + C_6 \Delta t^{k_0+k} \|u\|_{k_0+1}.
    \end{align}
    Altogether, it follows
    \begin{align*}
        \infnorm{\bar{U} - U^{(k)}} &\leq \infnorm{\bar{U} - U} + \infnorm{U - U^{(k)}}\\
        &\leq C_1 \Delta t^{M+1} \|u\|_{M+1} + \tilde{C}_4 \Delta t^{M+k+1} \|u\|_{M+1} + C_6 \Delta t^{k_0+k} \|u\|_{k_0+1}\\
        &= (C_1 + \tilde{C}_4 \Delta t^k) \Delta t^{M+1} \|u\|_{M+1} + C_6 \Delta t^{k_0+k} \|u\|_{k_0+1}.
    \end{align*}
    Since the step size $\Delta t$ is assumed to be sufficiently small, i.e.\ bounded above, the following estimate is valid
    \begin{align*}
       	C_1 + \tilde{C}_4 \Delta t^k \leq  C_5,
    \end{align*}
    which finally concludes the proof for \fref{equ:sdc_conv_lte}.
\end{proof}

\begin{remark}\label{rem:dx_sdc}
    Note that if the right-hand side of the ODE comes from a discretized PDE with a given spatial resolution, no additional restriction is posed.
    In this case, $C_3 = C_3(\Delta x^{-d}) = \tilde{C}_5\Delta x^{-d}$ for some constant $\tilde{C}_5$, so that the condition~\eqref{eq:conv_cond} becomes $\tilde{C}_5\Delta t < \Delta x^{d}$.
    Since we did not specify $Q_\Delta$ here, the SDC iterations can be explicit or implicit and it is natural to obtain such a restriction on the time-step size.
    Similar restrictions can be found in other convergence results for SDC, see e.g.~\cite{causley,hagstrom,huang}.
    Condition~\eqref{eq:conv_cond_II} then translates to
   	\begin{align*}
   		C_4 \geq \frac{\tilde{C}_1}{1-C_3 \Delta t} = \frac{\tilde{C}_6\Delta x^{-d}}{1-\tilde{C}_5\Delta x^{-d} \Delta t} = \frac{\tilde{C}_6}{\Delta x^d-\tilde{C}_5 \Delta t}
   	\end{align*}
   	for some constant $\tilde{C}_6$.
   	Assuming a fixed distance between $\Delta x^p$ and $\tilde{C}_5 \Delta t$ with $\Delta x^d-\tilde{C}_5 \Delta t = \delta$ we can write $C_4 = C_4(\delta^{-1})$ to indicate the dependence of $C_4$ on this distance and not $\Delta x^p$ or $\Delta t$ alone.
   	Thus, $C_4$ does not pose an additional restriction to the convergence of SDC.
   	Note that for $C_5$ we have $C_5 = C_5(\delta^{-(k+1)})$, so that the choice of $\delta$ can increase the constant in front of the $\Delta t^{k_0+k}$-term quite substantially, but it does not affect the $\Delta t^{M+1}$-term coming from the collocation problem itself. 
\end{remark}

\begin{remark}\label{rem:Lip_k}
    More generally, the constant $C_4$ depends on the Lipschitz constant $L$ of $f$ with
    \begin{align*}
        C_4 = C_4(k) \approx \tilde{C}\frac{L^k}{(1-L\dt)^k}\diff{\in\mathcal{O}(L^k)}
    \end{align*}
    for iteration $k$ and some constant $\tilde{C}$\diff{, if $L\dt$ is small enough}. 
    Thus, for $L>1$ the ``constant'' $C_4 = C_4(k)$ can grow significantly with $k$.
    When running numerical convergence studies for a fixed $\dt$, the gain in the error when going from one iterate $U^{(k)}$ to the next one $U^{(k+1)}$ may therefore be not as large as the asymptotic behavior in $\dt$ suggest.
    This is true for all results found in this paper.
\end{remark}

The theorem can be read as a convergence statement for SDC. In particular, the first estimation (\ref{equ:sdc_conv_iter}) shows that SDC, interpreted as an iterative method to solve the collocation problem, converges linearly to the solution of the collocation problem with a convergence rate of $\bigo(\Delta t)$ if $C_4\Delta t < 1$. The second part of the theorem, \fref{equ:sdc_conv_lte} shows that SDC, in the sense of a discretization method, converges with order $\min(k_0+k, M+1)$ to the solution of the initial value problem. In other words, the method gains one order per iteration, limited by the selected number of nodes used for discretization. 

We can now immediately extend this result by looking at the right endpoint of the single time interval (which, in our case, is equal to the last collocation node). There, the convergence rate is limited not by the number of collocation nodes, but by the order of the quadrature.

\begin{corollary}
	\label{cor:conv_last}
	Consider a generic initial value problem like (\ref{equ:ivp}) with a Lipschitz-continuous function $f$ on the right-hand side. Furthermore, let the solution of the initial value problem $u$ be $(2M)$-times continuously differentiable. 
	
	Then, if the step size $\Delta t$ is sufficiently small, the error of the $k$-th iterated of SDC, defined by \fref{equ:sdc_implicit}, at the last collocation node $u_M^{(k)}$, compared to the exact value at this point $u(\tau_M)$, can be bounded by
	\begin{align}
		\label{equ:conv_last_point}   
        \begin{split}
            \infnorm{u(\tau_m) - u_M^{(k)}} &\leq C_7 \Delta t^{2M} \max(\|u\|_{2M}, \|u\|_{M+1})\\
            &\quad+ C_8 \Delta t^{k_0+k}  \max(\|u\|_{k_0+1}, \|u\|_{M+1})
        \end{split}\\
        &= \bigo(\Delta t^{\min(k_0+k, 2M)}),\nonumber
	\end{align}
	where the constants $C_7$ and $C_8$ are independent of $\Delta t$, $k_0$ denotes the approximation order of the initial guess $U^{(0)}$ and $\|u\|_p$ is defined by $\infnorm{u^{(p)}}$.
\end{corollary}

\begin{proof}
    The proof mainly relies on the interpretation of the solution of the collocation problem evaluated at the last node $\tau_M$ as the result of a Radau method with $M$ stages. With this in mind, the well-known convergence, or in this case rather consistency, order of Radau methods yields the estimate \cite{hairer-wanner-2}
    \begin{align*}
    	\infnorm{u(\tau_M) - u_M} \leq \tilde{C}_1 \Delta t^{2M} \|u\|_{2M},
    \end{align*}
    where $\tilde{C}_1$ is independent of $\Delta t$. Here and in the following, temporary arising constants will again be denoted by symbols like $\tilde{C}_i$. However, they are separately defined and thus, do not correspond to the ones used in previous proofs.
    
    To use this estimation, we first have to apply the triangle inequality to the left-hand side of \fref{equ:conv_last_point}, in particular
    \begin{align}
    	\label{equ:proof_last_point_triangle}
    	\infnorm{u(\tau_m) - u_M^{(k)}} &\leq \infnorm{u(\tau_m) - u_M} + \infnorm{u_M - u_M^{(k)}}.
    \end{align}
    Then, with the definition of the vector $R \coloneqq (0,\dots,0,1) \in \real^{1 \times M}$ which, multiplied with another vector, only captures its last value, the second term on the right-hand side of the above equation can be transferred to
    \begin{align*}
        \infnorm{u_M - u_M^{(k)}} &= \infnorm{R U - R U^{(k)}} \leq \infnorm{R} \infnorm{U - U^{(k)}} = \infnorm{U - U^{(k)}}\\
        &\le \tilde{C}_2 \Delta t^{M+k+1} \norm{u}_{M+1} + C_6 \Delta t^{k_0+k}  \norm{u}_{k_0+1},
    \end{align*}
    where the last estimate comes from \fref{equ:proof_sdc_err_coll} in the proof of \fref{theo:sdc_conv}.
    Finally, by inserting all these results in \fref{equ:proof_last_point_triangle}, we obtain
    \begin{align*}
        \infnorm{u(\tau_m) - u_M^{(k)}} &\leq \infnorm{u(\tau_m) - u_M} + \infnorm{U - U^{(k)}}\\
        &\leq \tilde{C}_1 \Delta t^{2M} \|u\|_{2M} + \tilde{C}_2 \Delta t^{M+k+1} \|u\|_{M+1} + C_5 \Delta t^{k_0+k}  \|u\|_{k_0+1}.
    \end{align*}
    Note that the leading order of this term is essentially independent of the summand corresponding to $\Delta t^{M+k+1}$. For $\Delta t$ small enough, this result can be seen by a case analysis for $k$. For $k+1 \geq M$, the considered summand is dominated by $\Delta t^{2M} \geq \Delta t^{M+k+1}$ and thus can be disregarded in terms of leading order analysis. In the other case, i.e.\ for $k+1< M$, the considered summand will, however, be greater than the one of order $2M$. Therefore, we will instead compare it to $\Delta t^{k_0+k}$ in this case. Since the number of collocation nodes $M$ is usually chosen to be greater than the approximation order $k_0$ of the initial guess, the relation $\Delta t^{k_0+k} \geq \Delta t^{M+k+1}$ applies and hence, ${k_0+k}$ will be the leading order for $k < M$. Thus, the considered summand $\Delta t^{M+k+1}$ is again dominated by another term and can be disregarded concerning the overall asymptotic behavior. These considerations consequently lead to the following estimation
    \begin{align*}
        \infnorm{u(\tau_m) - u_M^{(k)}} \leq \tilde{C}_3 \Delta t^{2M} \max(\|u\|_{2M}, \|u\|_{M+1}) + C_5 \Delta t^{k_0+k}  \|u\|_{k_0+1}
    \end{align*}
    for $ k \geq M$ and
    \begin{align*}
        \infnorm{u(\tau_m) - u_M^{(k)}} \leq \tilde{C}_1 \Delta t^{2M} \|u\|_{2M} + \tilde{C}_4 \Delta t^{k_0+k}   \max(\|u\|_{k_0+1}, \|u\|_{M+1})
    \end{align*}
    for $k < M$, which can be combined to
    \begin{align*}
        \infnorm{u(\tau_m) - u_M^{(k)}} &\leq C_7 \Delta t^{2M} \max(\|u\|_{2M}, \|u\|_{M+1})\\
        &\quad+ C_8 \Delta t^{k_0+k}  \max(\|u\|_{k_0+1}, \|u\|_{M+1}),
	\end{align*}
    concluding the proof.
\end{proof}


With this corollary, it can be concluded that SDC, in the sense of a single-step method to solve ODEs, is consistent of order $\min(k_0+k-1, 2M-1)$. To extend this result towards a statement on the convergence order of the method an additional proof of its stability is needed. 
The following theorem provides an appropriate result for SDC.

\begin{theorem}
    \label{theo:sdc_stab}
    Consider a generic initial value problem like (\ref{equ:ivp}) with a Lipschitz-continuous function $f$ on the right-hand side. 
    If the step size $\Delta t$ is sufficiently small and an appropriate initial guess is used, the SDC method, defined by \fref{equ:sdc}, is stable.
\end{theorem}
\begin{proof}
    As usual for single-step methods, we will prove the Lipschitz continuity of the increment function of SDC in order to prove the stability of the method.
    
    A general single-step method is defined by the formula 
    \begin{align*}
        u_{n+1} = u_n + \dt \phi(u_n),    
    \end{align*}
    where $u_n$ denotes the approximation at the time step $t_n$ and $\phi(u_n)$ is the increment function. Our aim now is to identify the specific increment function $\phi$ corresponding to SDC and to, subsequently, show its Lipschitz continuity, i.e. prove the validity of $\abs{\phi(u_n) - \phi(v_n)} \leq L_\phi \abs{u_n - v_n}$.
    
    First, note that the $k$-th iterated of SDC $u_m^{(k)}$ at an arbitrary collocation node $\tau_m$ ($1 \leq m \leq M$) can be written as
    \begin{align*}
        u_m^{(k)} &= u_n + r_m^{(k)} \mbox{ with}\\
        r_m^{(k)} &\coloneqq \dt (Q_\Delta F(U_n+r^{(k)}))_m + \dt ((Q - Q_{\Delta}) F(U_n + r^{(k-1)}))_m,\\
    r^{(k)} &\coloneqq (r_1^{(k)}, \dots, r_M^{(k)})^T \mbox{ and } U_n = (u_n, \dots, u_n)^T
    \end{align*}
    for $k\ge1$ according to a line-wise consideration of \fref{equ:sdc}, where the subscript $m$ denotes the $m$th line of the vectors. Consequently, the corresponding approximation at the time step $t_{n+1} = \tau_M$ can be written as 
    \begin{gather*}
        u_{n+1} = u_M^{(k)} = u_n + r_M^{(k)} \eqqcolon u_n + \dt \phi^{(k)}(u_n)\\
        \mbox{with } \phi^{(k)}(u_n) = \frac{1}{\dt} r_M^{(k)}.
    \end{gather*}
    Hence, we have found an appropriate, albeit implicit definition for the increment function $\phi^{(k)}$ of SDC.
    
    As a second step, it follows an investigation on the Lipschitz continuity of this function. For that, we start by noting that
    \begin{align}
        \label{equ:rel_phi_r}
        \abs{\phi^{(k)}(u_n) - \phi^{(k)}(v_n)} &= \frac{1}{\dt} \abs{r_M^{(k)} - s_M^{(k)}} \leq \frac{1}{\dt} \infnorm{r^{(k)} - s^{(k)}},
    \end{align}
    where the $r$-terms belong to $u_n$ and the $s$-terms to $v_n$. Now, we will further analyze the term $\norm{r^{(k)} - s^{(k)}}$. With the insertion of the corresponding definitions and an application of the triangle inequality, it follows
    \begin{align*}
        \norm{r^{(k)} - s^{(k)}} &\leq \norm{\Delta t Q_\Delta (F(U_n + r^{(k)}) - F(V_n + s^{(k)}))}\\
        &\quad+ \norm{\Delta t (Q - Q_\Delta) (F(U_n + r^{(k-1)}) - F(V_n + s^{(k-1)}))}.
    \end{align*}
    The use of \fref{lem:estimation} and a reapplication of the triangle inequality further yield
    \begin{align*}
        \norm{r^{(k)} - s^{(k)}} &\leq \tilde{C}_1 \dt(\abs{u_n - v_n} + \norm{r^{(k)} - s^{(k)}} + \norm{r^{(k-1)} - s^{(k-1)}}).
    \end{align*}
    Continuing with the same trick as in the proof of \fref{theo:sdc_conv}, namely a subtraction of $\tilde{C}_1 \dt \norm{r^{(k)} - s^{(k)}}$ and a subsequent division by $1- \tilde{C}_1 \dt$, we get
    \begin{align*}
        \norm{r^{(k)} - s^{(k)}} &\leq \frac{\tilde{C}_1}{1- \tilde{C}_1 \dt} \dt (\abs{u_n - v_n} + \norm{r^{(k-1)} - s^{(k-1)}}).
    \end{align*}
    If the step size $\dt$ is sufficiently small, the following estimation applies
    \begin{align*}
        \frac{\tilde{C}_1}{1- \tilde{C}_1 \dt} &\leq \tilde{C}_2
    \end{align*}
    and a subsequent iterative insertion further yields
    \begin{align*}
        \norm{r^{(k)} - s^{(k)}} \leq \tilde{C}_3 \sum_{l = 1}^{k} \dt^l \abs{u_n-v_n} + \tilde{C}_4 \dt^{k} \norm{r^{(0)} - s^{(0)}}.
    \end{align*}
    With the insertion of this result in \fref{equ:rel_phi_r} above, it finally follows
    \begin{align*}
    \abs{\phi^{(k)}(u_n) - \phi^{(k)}(v_n)} \leq C \sum_{l = 0}^{k-1} \dt^l \abs{u_n-v_n} + C \dt^{k-1} \norm{r^{(0)} - s^{(0)}}.
    \end{align*}
    The value of $\norm{r^{(0)} - s^{(0)}}$ depends on the initial guess for the SDC iterations. If, for example, the value at the last time step is used as the initial guess for all collocation nodes, i.e.\ $U^{(0)} = U_n$ and $V^{(0)} = V_n$, we get $\norm{r^{(0)} - s^{(0)}} = 0 - 0 = 0$. If, by contrast, the initial guess is chosen to be zero, it follows $\norm{r^{(0)} - s^{(0)}} = \abs{u_n - v_n}$. Both variants, however, guarantee that $\abs{\phi^k(u_n) - \phi^k(v_n)} \leq C \abs{u_n - v_n}$ which was to be shown. 
\end{proof}


\begin{remark}
    Note that the assumed upper bound for the step size $\dt$ in the previous theorem is the same as the one in \fref{cor:conv_last}, describing the consistency of SDC. Hence, there is no additional restriction for the convergence of the method.
\end{remark}

Together with the last theorem, \fref{cor:conv_last} can be extended towards a convergence theorem for SDC regarded in the context of single-step methods to solve ODEs. Specifically, the proven stability of the method allows a direct transfer of the order of consistency to the order of convergence. Consequently, it follows that SDC, in the sense of a single-step method, converges with order $\min(k_0+k-1, 2M-1)$.

In the next chapter, MLSDC, a multi-level extension of SDC, is described. It is motivated by the assumption that additional iterations on a coarser level may increase the order of accuracy while keeping the costs rather low. In the next chapter we will investigate whether this assumption holds true, i.e. if the convergence order is indeed increased by the additional execution of relatively low-cost iterations on the coarse level.

\section{Multi-Level Spectral Deferred Corrections}
\label{sec:mlsdc}


Multi-level SDC (MLSDC) is a method that uses a multigrid-like approach to solve the collocation problem \eqref{equ:coll_prob}. It is an extension of SDC in which the iterations, now called ``sweeps'' in this context, are computed on a hierarchy of levels and the individual solutions are coupled in the same manner as used in the full approximation scheme (FAS) for non-linear multigrid methods.

The different levels are typically created by using discretizations of various resolutions. In this paper, only the two-level algorithm is considered. For this purpose, let $\Omega_h$ denote the fine level and $\Omega_H$ the coarse one. Then, $U_h$ denotes the discretized vector on $\Omega_h$. Furthermore, $C_h$, $F_h$ and $Q_h$ are the discretizations of the operators and the quadrature matrix. Likewise, $U_H$, $C_H$, $F_H$ and $Q_H$ represent the corresponding values for the discretization parameter $H$.

Here, we will consider two coarsening strategies. The first one is a re-discretization in time at the collocation problem, i.e.\ a reduction of collocation nodes. The second possibility, only applicable if a partial differential equation has to be solved, is a re-discretization in space, i.e.\ the use of less variables for the conversion into an ODE.

Since it is necessary to perform computations on different levels, a method to transfer vectors between the individual levels is needed. For this purpose, let $I_H^h$ denote the operator that transfers a vector from the coarse level $\Omega_H$ to the fine level $\Omega_h$. This operator is called the interpolation operator. $I_h^H$, on the other hand, shall represent the operator for the reverse direction. It is called the restriction operator. Both operators together are called transfer operators.


In detail, the MLSDC two-level algorithm consists of these four steps. Note that for better readability, the enlargements of the matrices $Q$ and $Q_\Delta$ by applying the Kronecker product with the identity matrix are no longer indicated.
\begin{enumerate}
    \item Compute the $\tau$-correction as the difference between coarse and fine level:
            \begin{align}\begin{split}
                \label{equ:tau-korr}
                \tau &= C_H(I_h^H U_h^{(k)}) - I_h^H C_h(U_h^{(k)})\\
                &= I_h^H ( \Delta t Q_h F_h (U_h^{(k)}) ) - \Delta t Q_H F_H(I_h^H U_h^{(k)}).
            \end{split}\end{align}
    \item Perform an SDC sweep to approximate the solution of the modified collocation problem on the coarse level
            \begin{equation}
                \label{equ:qu:coll_coarse}
                C_H(U_H) = U_{0,H} + \tau
            \end{equation}
            on $\Omega_H$, beginning with $I_h^H U_h^{(k)}$:
            \begin{align}\begin{split}
                \label{equ:mlsdc-alg1}
                U_H^{(k+\half)} &= U_{0,H} + \tau + \Delta t Q_{\Delta,H} F_H(U_H^{(k+\half)})\\
                &\quad+ \Delta t (Q_H - Q_{\Delta,H}) F_H(I_h^H U_h^{(k)}).
            \end{split}\end{align}
    \item Compute the coarse level correction:
            \begin{align}
                \label{equ:mlsdc-alg2}
                U_h^{(k+\half)} = U_h^{(k)} + I_H^h (U_H^{(k+\half)} - I_h^H U_h^{(k)}).
            \end{align}
    \item Perform an SDC sweep to approximate the solution of the original collocation problem
            \begin{equation*}
                C(U_h) = U_{0,h}
            \end{equation*}
            on $\Omega_h$, beginning with $U_h^{(k+\half)}$:
            \begin{align}\begin{split}
                \label{equ:mlsdc-alg3}
                U_h^{(k+1)} &= U_{0,h} + \Delta t Q_{\Delta,h} F_h(U_h^{(k+1)})\\
                &\quad+ \Delta t (Q_h - Q_{\Delta,h}) F_h(U_h^{(k+\half)}).
            \end{split}\end{align}
\end{enumerate}

\subsection{A first convergence proof}

Here, we will extend the existing convergence proof for SDC, as presented in \fref{theo:sdc_conv}, to prove the convergence of its multi-level extension MLSDC. The following theorem provides an appropriate convergence statement. In the proof, we use very similar ideas as in the one for the convergence of SDC.        

\begin{theorem}
	\label{theo:mlsdc_conv_gen}
	Consider a generic initial value problem like (\ref{equ:ivp}) with a Lipschitz-continuous function $f$ on the right-hand side. 
    
    If the step size $\Delta t$ is sufficiently small, MLSDC converges linearly to the solution of the collocation problem with a convergence rate in $\bigo(\dt)$, i.e. the following estimate for the error of the $k$-th iterated $U_h^{(k)}$ of MLSDC compared to the solution of the collocation problem $U_h$ is valid:
	\begin{align}
	    \label{equ:mlsdc_conv_gen_iter}
    	\infnorm{U_h - U_h^{(k)}} \leq C_9 \Delta t \infnorm{U_h - U_h^{(k-1)}},
	\end{align}
	where the constant $C_9$ is independent of $\Delta t$.
   	
   	If, additionally, the solution of the initial value problem $u$ is $(M+1)$-times continuously differentiable, the LTE of MLSDC compared to the solution of the ODE can be bounded by
   	\begin{align}
	    \label{equ:mlsdc_conv_gen_lte}
		\|\bar{U} - U_h^{(k)}\|_\infty &\leq C_{10} \Delta t^{k_0+k} \|u\|_{k_0+1} 
		+ C_{11} \Delta t^{M+1} \|u\|_{M+1}\\
        &= \bigo(\Delta t^{\min(k_0+k, M+1)}),
	\end{align}
	where the constants $C_{10}$ and $C_{11}$ are independent of $\Delta t$, $k_0$ denotes the approximation order of the initial guess $U^{(0)}$ and $\|u\|_p$ is defined by $\infnorm{u^{(p)}}$.
\end{theorem}
\begin{proof}
    For better readability, the maximum norm $\infnorm{\cdot}$ is denoted with the simple norm $\norm{\cdot}$ within this proof.
    \renewcommand{\infnorm}[1]{\lVert#1\rVert}
    Besides, we consider $U_h^{(k+1)}$ instead of $U_h^{(k)}$ here, in order to enable consistent references to the definition of the MLSDC algorithm above.
    
    As the last step of an MLSDC iteration, in particular \fref{equ:mlsdc-alg3}, corresponds to an SDC iteration, we can use \fref{theo:sdc_conv} to get an initial error estimation. Keeping in mind that the SDC iteration is based on $U_h^{(k+\half)}$ as initial guess here, the application of the mentioned theorem yields the estimation
    \begin{align}
        \label{equ:u-1_rek_u-half}
        \infnorm{U_h - U_h^{(k+1)}} \leq C_4 \Delta t \infnorm{U_h - U_h^{(k+\half)}},
    \end{align}
    if the step size $\Delta t$ is sufficiently small.
    
    Now, the expression on the right-hand side of the above equation will be further examined. In this context, the definition of an MLSDC iteration, in particular \fref{equ:mlsdc-alg2}, yields
    \begin{align*}
        \infnorm{U_h - U_h^{(k+\half)}} &= \infnorm{U_h - U_h^{(k)} - I_H^h \left( U_H^{(k+\half)} - I_h^H U_h^{(k)} \right)}\\
        &= \infnorm{U_h - U_h^{(k)} - I_H^h \left( U_H^{(k+\half)} + U_H - U_H - I_h^H U_h^{(k)} \right)}\\
        &= \infnorm{(I - I_H^h I_h^H) (U_h - U_h^{(k)}) + I_H^h ( U_H - U_H^{(k+\half)} )},
    \end{align*}
    where in the last step we used the identity $U_H = I_h^H U_h$ which applies in consequence of the $\tau$-correction stemming from the usage of FAS. We get
    \begin{align}
        \label{equ:u-half_Ihh}
        \infnorm{U_h - U_h^{(k+\half)}} &\leq \infnorm{(I - I_H^h I_h^H) (U_h - U_h^{(k)})} + \infnorm{I_H^h ( U_H - U_H^{(k+\half)} )}\\
        &\leq \tilde{C}_1 \infnorm{U_h - U_h^{(k)}} + \tilde{C}_2 \infnorm{U_H - U_H^{(k+\half)}},
        \label{equ:u-half}
    \end{align}
    where here and in the following, temporary arising constants will again be denoted by symbols like $\tilde{C}_i$.
    
    Now, it follows a further investigation of the newly emerged term, in particular the second summand of \fref{equ:u-half}. An insertion of the corresponding definitions, namely equations \eqref{equ:coll_prob}, \eqref{equ:qu:coll_coarse} and (\ref{equ:mlsdc-alg1}), together with the application of the triangle inequality and \fref{lem:estimation} yields
    \begin{align*}
        \begin{split}
        \infnorm{U_H - U_H^{(k+\half)}} 
        &= \infnorm{\Delta t Q_H (F_H(U_H) - F_H(I_h^H U_h^{(k)}))\\
        &\quad\quad + \Delta t Q_{\Delta,H} (F_H(I_h^H U_h^{(k)}) - F_H(U_H^{(k+\half)}))} 
        \end{split}\\
        &\leq C_{2,H} \Delta t \infnorm{U_H - I_h^H U_h^{(k)}} + C_{3,H} \Delta t \infnorm{I_h^H U_h^{(k)} - U_H^{(k+\half)}} \nonumber\\
        &\leq \tilde{C}_3 \Delta t \infnorm{U_H - I_h^H U_h^{(k)}} + C_{3,H} \Delta t \infnorm{U_H - U_H^{(k+\half)}} \nonumber.
    \end{align*}
    Subtracting $C_{3,H} \Delta t \infnorm{U_H - U_H^{(k+\half)}}$ from both sides and dividing by $1-C_{3,H} \Delta t$, results in
    \begin{align*}
        \infnorm{U_H - U_H^{(k+\half)}} &\leq \frac{\tilde{C}_3}{1-C_{3,H} \Delta t} \Delta t \infnorm{U_H - I_h^H U_h^{(k)}}.
    \end{align*}
    With the same argumentation as above, given a sufficiently small step size, it follows
    \begin{align}
        \infnorm{U_H - U_H^{(k+\half)}} &\leq \tilde{C}_4 \Delta t \infnorm{U_H - I_h^H U_h^{(k)}} \nonumber\\
        &\leq \tilde{C}_5 \Delta t \infnorm{U_h - U_h^{(k)}},
        \label{equ:u-half_H}
    \end{align}
    where in the last step we used the identity $U_H = I_h^H U_h$.
    
    By inserting \fref{equ:u-half} and this result subsequently into \fref{equ:u-1_rek_u-half}, we obtain
    \begin{align*}
        \infnorm{U_h - U_h^{(k+1)}} &\leq C_4 \Delta t \infnorm{U_h - U_h^{(k+\half)}}\\
        &\leq C_4 \Delta t (\tilde{C}_1 \infnorm{U_h - U_h^{(k)}} + \tilde{C}_2 \infnorm{U_H - U_H^{(k+\half)}})\\
        &\leq \tilde{C}_6 \Delta t \infnorm{U_h - U_h^{(k)}} + \tilde{C}_7 \Delta t ( \tilde{C}_5 \Delta t \infnorm{U_h - U_h^{(k)}})\\
        &= (\tilde{C}_6 + \tilde{C}_8 \Delta t) \Delta t \infnorm{U_h - U_h^{(k)}}.
    \end{align*}
    Since the step size $\Delta t$ is assumed to be sufficiently small, i.e.\ bounded above, the following estimate is valid
    \begin{align*}
        \tilde{C}_6 + \tilde{C}_8 \Delta t \leq C_9,
    \end{align*}
    which concludes the proof for \fref{equ:mlsdc_conv_gen_iter}.
    
	The proof of \fref{equ:mlsdc_conv_gen_lte} is similar to the one of \fref{equ:sdc_conv_lte} in \fref{theo:sdc_conv}, using the previous result.
\end{proof}

\begin{remark}
    As before, if the Lipschitz constant of $f$ depends on the spatial resolution, then constants $\tilde{C}_3$ and $C_{3,H}$ do as well.
    In~\eqref{equ:u-half_H}, constant $\tilde{C}_5$ then only depends on the original condition $1 - C_{3,H} < 1$, which is just the condition we know from SDC, but with a scaled constant if spatial coarsening is applied.
    As in remark~\ref{rem:dx_sdc} we can write $\tilde{C}_5 = \tilde{C}_5(\delta^{-1})$. 
    This only affects $C_{10}$ in \fref{equ:mlsdc_conv_gen_lte}, where now $C_{10} = C_{10}(\delta^{-(k+1)})$.
\end{remark}

As \fref{theo:sdc_conv}, this theorem can also be read as a convergence statement. It shows that MLSDC, interpreted as an iterative method solving the collocation problem, converges linearly with a convergence rate in $\bigo(\Delta t)$ if $C_9 \Delta t < 1$. Moreover, the second part of the theorem shows that MLSDC, in the sense of a discretization method for ODEs, converges with order $\min(k_0+k, M+1)$.

\begin{remark}
    \label{rem:mlsdc_analogous}
    The results regarding consistency and stability of SDC, namely \fref{cor:conv_last} and \fref{theo:sdc_stab}, can be easily adapted for MLSDC. Analogous to SDC, it can be proven that the error at the last collocation node can be bounded by $\bigo(\dt^{\min(k_0+k, 2M)})$ and the increment function of MLSDC is Lipschitz continuous for $\dt$ small enough.
\end{remark}

Although \fref{theo:mlsdc_conv_gen} is the first general convergence theorem for MLSDC, its statement is rather disappointing: It merely establishes that MLSDC converges as least as fast as SDC, although more work is done per iteration.


A deeper look into the proof of the theorem gives an idea on the cause for the rather unexpected low convergence order. In particular, it is the estimate leading to \fref{equ:u-half} which is responsible for this issue. This equation implies that 
\begin{align*}
    \norm{U_h - U_h^{(k+\half)}} \leq C \norm{U_h - U_h^{(k)}},
\end{align*} 
which essentially means that the additional iteration on the coarse level does not gain any additional order in $\Delta t$ compared to the previously computed iterated on the fine level $U_h^{(k)}$. More specifically, it is the estimation 
\begin{align*}
    \infnorm{(I-I_H^h I_h^H)(U_h - U_h^{(k)})} \leq C \infnorm{U_h - U_h^{(k)}},
\end{align*}
in equation \eqref{equ:u-half} which leads to this result.


		
Thus, a possibly superior behavior of MLSDC seems to depend on the magnitude of $\infnorm{(I - I_H^h I_h^H)e_h}$ with $e_h = U_h - U_h^{(k)} \in \mathbb{C}^{MN}$ which describes the difference between an original vector and the one which results from restricting and interpolating it. Consequently, the term can be interpreted as the quality of the approximation on the coarse level or the accuracy loss it causes, respectively. In the following, this term will be examined in detail, resulting in a new theorem for the convergence of MLSDC with a higher convergence order but additional assumptions which have to be met. 

\subsection{An improved convergence result}\label{ssec:improved_mlsdc}

For this purpose, we will mainly focus on a specific coarsening strategy here, in particular coarsening in space. The differences occurring from coarsening in time will be discussed at the end of this section. Moreover, we will focus on particular methods used for the transfer operators. For $I_H^h$ we consider a piece-wise Lagrange interpolation of order $p$. This means that instead of using all $N_H$ available values to approximate the value at a particular point $x_i \in \Omega_h$, only its $p$ neighbors are taken into account for this purpose. Hence, $I_H^h$ corresponds to the application of a $p$-th order Lagrange interpolation for each point. For the restriction operator $I_h^H$, on the other hand, we consider simple injection. Thereby, we can mainly focus on the interpolation order and disregard the restriction order.

The following lemma now provides an appropriate estimation for the considered term $\infnorm{(I - I_H^h I_h^H)e_h}$.

\begin{lemma}
	\label{lem:ihh_gen}
    Let $E \coloneqq (E_m)_{1 \leq m \leq M}$ denote the remainder of the truncated inverse discrete Fourier transformation of $U_h - U_h^{(k)}$, i.e.
    \begin{align*}
        E_m \coloneqq \sum_{\ell=N_0}^{N-1} \abs{c_{m,\ell}},\quad m = 1, \dots, M.
    \end{align*}
    for some cutoff index $N_0 \le N$ and $c_{m,\ell}\in\mathbb{C}$ being the Fourier coefficients.
    Then, the following estimate for this error is valid:
    \begin{align*}
        \infnorm{(I - I_H^h I_h^H)(U_h - U_h^{(k)})} \leq (C_{11} \Delta x^p + C_{12}(E)) \infnorm{U_h - U_h^{(k)}},
    \end{align*}
    where $I \equiv I_{MN}$ denotes the identity matrix of size $MN$, $I_H^h$ is the piece-wise spatial Lagrange interpolation of order $p$ and $I_h^H$ denotes the injection operator. Furthermore, $\Delta x \equiv \Delta x_H$ is defined as the resolution in space on the coarse level $\Omega_H$ of MLSDC.
\end{lemma}
\begin{proof}
	\renewcommand{\infnorm}[1]{\lVert#1\rVert_\infty}
	First of all, we need to introduce some definitions. 
	For ease of notation, the considered error vector $U_h - U_h^{(k)}$ will be denoted by $e_h$ within the proof. 
	In detail, the following definition is used:
    \begin{align*}
        e_{m,n} = e_{h,n}(\tau_m) = u_{h,n}(\tau_m)- u_{h,n}^{(k)}(\tau_m), \quad \forall m=1,\dots,M,\; n=1,\dots,N,
    \end{align*}
    where $m$ denotes the temporal index identifying the particular collocation node $\tau_m$ and $n$ represents the spatial index referring to some discretized points $(x_n)_{1 \leq n \leq N}$ within the considered interval in space $[0,S]$. Additionally, we assume the spatial steps to be equidistant here.
    
    Another definition needed for the proof is $g_{m,p}(x)$. It denotes the Lagrangian interpolation polynomial of order $p$ for the restricted vector $I_h^H e_h$ for each point in time $\tau_m$, $m=1, \dots, M$. As the two levels $\Omega_h$ and $\Omega_H$ differ in their spatial resolution, the transfer operators are applied at the spatial axis and thus can be considered separately for each component $e_h(\tau_m)$. More specifically, the restriction operator, corresponding to simple injection according to the assumptions, omits several values of $e_h(\tau_m) \in \compl^N$ resulting in $(I_h^H e_h)(\tau_m) \in \compl^{N_H}$, where $N_H$ denotes the number of degrees of freedom at the coarse level. The subsequent application of the interpolation operator at this vector then leads to the $M$ interpolating polynomials $(g_{m,p}(x))_{1 \leq m \leq M}$ with $p$ referring to their order of accuracy.
    
    Having introduced these notations, the considered term can be written as
    \begin{align*}
    (I - I_H^h I_h^H) (U_h - U_h^{(k)}) = (I - I_H^h I_h^H) e_h = \left(e_{m,n} - g_{m,p}(x_n)\right)_{\substack{1 \leq m \leq M,\\1 \leq n \leq N}},
    \end{align*}
    so that we can focus on $\abs{e_{m,n} - g_{m,p}(x_n)}$.
    
    Since $g_{m,p}(x)$ partially interpolates the points $(e_{m,n})_{1 \leq n \leq N}$, it seems very reasonable to use the general error estimation of Lagrangian interpolation to determine an estimate for the considered term. However, there is a crucial issue: The corresponding error bound, generally defined in \cite{bartels,interp-err} by
    \begin{align}
        \label{equ:interpol_error}
        \max_{x \in [a,b]} \abs{f(x) - g_p(x)} \leq \frac{\Delta x^p}{4p} \abs{f^{(p)} (\xi)}, \quad \xi \in [a,b]
    \end{align}
    apparently depends on the function $f$ from which the interpolation points are obtained. In our case, namely $\abs{e_{m,n} - g_{m,p}(x_n)}$, we do not have a specific function directly available to which the error components $e_{m,n}$ correspond. 
    
    To fill this gap, we will now derive an appropriate function for this purpose, using a continuous extension of the inverse Discrete Fourier Transformation (iDFT).        
    The iDFT at the spatial axes for the points $e_{m,n}$ is given by
    \begin{align*}
        e_{m,n} = \frac{1}{\sqrt{N}} \sum_{\ell=0}^{N-1} c_{m,\ell} \exp{i \frac{2\pi}{N} (n-1) \ell}, \quad m=1, \dots, M, \quad n=1, \dots, N
    \end{align*}
    with $c_{m,\ell}$ as Fourier coefficients and $i$ symbolizing the imaginary unit \cite{fourier}. 
    
    A continuous extension $\tilde{e}_m(x), x \in [0,S]$, on the whole spatial interval can then be derived by enforcing $\tilde{e}_m(x_n) \diff{\ =\ } e_{m,n}$ for all $m$ and $n$.
    With the transformation $x_n = \frac{S}{N} (n-1)$, i.e.\ $n = \frac{N}{S} x_n +1$, implied by an equidistant spatial discretization, it follows
    \begin{align*}
        \tilde{e}_m(x_n) &= \frac{1}{\sqrt{N}} \sum_{\ell=0}^{N-1} c_{m,\ell} \exp{i \frac{2\pi}{S}\ell x_n}
    \end{align*}
    and hence
    \begin{align*}
        \tilde{e}_m(x) \coloneqq \frac{1}{\sqrt{N}} \sum_{\ell=0}^{N-1} c_{m,\ell} \exp{i \frac{2\pi}{S} \ell x}, \quad x \in [0, S], \quad m=1, \dots, M.
    \end{align*}
    
    Consequently, we have found a function describing the points $e_{m,n}$. Thus, an interpretation of $g_{m,p}(x)$ as interpolation polynomial of $p$ points stemming from $\tilde{e}_m(x)$ is possible and the error estimation for Lagrange interpolation, presented in \fref{equ:interpol_error}, can now be applied. As a result, we get
    \begin{align}
        \abs{e_{m,n} - g_{m,p}(x_n)} = \abs{\tilde{e}_m(x_n) - g_{m,p}(x_n)} &\leq \frac{\Delta x^p}{4p} \abs{\tilde{e}_m^{(p)} (\xi)}
    \end{align} 
    with $\xi \in [0,S]$.
    
    With the insertion of the definition of $\tilde{e}_m(x)$ and its $p$th derivative, it follows
    \begin{align*}
    	\abs{\tilde{e}_m(x_n) - g_{m,p}(x_n)} &\leq \frac{\Delta x^p}{4p} \abs*{\frac{1}{\sqrt{N}} \sum_{\ell=0}^{N-1} c_{m,\ell} \left(i \frac{2\pi}{S}\ell \right)^p \exp{i \frac{2\pi}{S}\ell \xi}}\\
    	&= \frac{1}{\sqrt{N}} \frac{\Delta x^p}{4p} \left(\frac{2\pi}{S}\right)^p \abs*{\sum_{\ell=0}^{N-1} c_{m,\ell} \ell^p \exp{i \frac{2\pi}{S}\ell \xi}},
    \end{align*}
    so that with $\tilde{C}(p) \coloneqq  \frac{1}{4p} \left(\frac{2\pi}{S}\right)^p$ we have
    \begin{align*}
    	\abs{\tilde{e}_m(x_n) - g_{m,p}(x_n)} &\leq \frac{1}{\sqrt{N}} \tilde{C}(p) \Delta x^p \sum_{\ell=0}^{N-1} \abs*{c_{m,\ell} \ell^p \exp{i \frac{2\pi}{S}\ell \xi}}\\
    	&\leq \frac{1}{\sqrt{N}} \tilde{C}(p) \Delta x^p \sum_{\ell=0}^{N-1} \abs{c_{m,\ell}} \ell^p.
    \end{align*}
    We now choose $N_0\le N$ such that 
    \begin{align*}
        \sum_{\ell=0}^{N_0-1} \abs{c_{m,\ell}} \geq \epsilon_m > 0 \quad \forall m
    \end{align*}
    for given $\epsilon_m$. Note that this is not possible if there exists an $m$ for which the coefficients $(c_{m,\ell})_{\ell = 0, \dots, N-1}$ are all $0$. This would imply that for this particular $m$ the error $e_{m,n}$ equals to $0$ for all $n = 1, \dots, N$. However, the respective error $e_h(\tau_m)$ would then be $0$ and could just be disregarded in our particular context as it does not have an impact on the considered maximum norm. Therefore, the assumption does not lead to a loss of generality. Finally, we define the remainders of the sums as
    \begin{align*}
        E_m \coloneqq \sum_{\ell=N_0}^{N-1} \abs{c_{m,\ell}},\quad m = 1, \dots, M.
    \end{align*}
    
    Now, the sum in the previous estimation will be split at $N_0$ resulting in the following estimate:
    \begin{align*}
        \abs{\tilde{e}_m(x_n) - g_{m,p}(x_n)} &\leq \frac{1}{\sqrt{N}} \tilde{C}(p) \Delta x^p \left( \sum_{\ell=0}^{N_0-1} \abs{c_{m,\ell}} \ell^p + \sum_{\ell=N_0}^{N-1} \abs{c_{m,\ell}}  \ell^p \right).
    \end{align*}
    With the simple estimation $\ell \leq N_0$ for the first sum and $\ell \leq N$ for the second one, the following formulation using the definition of the remainder $E_m$ is obtained:
    \begin{align*}
    	\abs{\tilde{e}_m(x_n) - g_{m,p}(x_n)} &\leq \frac{1}{\sqrt{N}} \tilde{C}(p) \Delta x^p \left( N_0^p \sum_{\ell=0}^{N_0-1} \abs{c_{m,\ell}} + N^p \sum_{\ell=N_0}^{N-1} \abs{c_{m,\ell}} \right)\\
    	&= \frac{1}{\sqrt{N}} \tilde{C}(p) \Delta x^p \left( N_0^p \sum_{\ell=0}^{N_0-1} \abs{c_{m,\ell}} + N^p E_m \right).
    \end{align*}
    
    Now, we will have a look at the norm of the whole vector $(I - I_H^h I_h^H)e_h$, but instead of the maximum norm, we first consider the squared 2-norm given by
    \begin{align*}
    	\norm{(I - I_H^h I_h^H) e_h}_2^2 &= \sum_{n=1}^N \sum_{m=1}^M \abs{\tilde{e}_m(x_n) - g_{m,p}(x_n)}^2.
    \end{align*}
	By the insertion of the previous estimation, it follows
    \begin{align*}
    	\norm{(I - I_H^h I_h^H) e_h}_2^2 &\leq \sum_{n=1}^N \sum_{m=1}^M \left[ \frac{1}{\sqrt{N}} \tilde{C}(p) \Delta x^p \left( N_0^p \sum_{\ell=0}^{N_0-1} \abs{c_{m,\ell}} + N^p E_m \right) \right]^2.
    \end{align*}
    Since the summands are independent of the running index $n$, the equation can be simplified to
    \begin{align}
        \label{equ:tmp1}
    	\norm{(I - I_H^h I_h^H) e_h}_2^2 &\leq \tilde{C}(p)^2 \Delta x^{2p} \sum_{m=1}^M \left( N_0^p \sum_{\ell=0}^{N_0-1} \abs{c_{m,\ell}} + N^p E_m \right)^2.
    \end{align}
    Now, each inner summand can be written as
    \begin{align}
        \label{equ:tmp2}
        \left( N_0^p \sum_{\ell=0}^{N_0-1} \abs{c_{m,\ell}} + N^p E_m \right)^2 &= 
    	N_0^{2p} \left( \sum_{\ell=0}^{N_0-1} \abs{c_{m,\ell}} \right)^2 + 
    	2 N_0^p N^p E_m \sum_{\ell=0}^{N_0-1} \abs{c_{m,\ell}} \nonumber \\
    	&\quad+ N^{2p} E_m^2\\
    	&=: S_1 + S_2 + S_3.\nonumber
    \end{align}
    While $S_1$ is an intended component (it contains the squared sum of $c_{m,\ell}$ which we will need to get back to the norm of $e_h$), the other two summands $S_2$ and $S_3$ are inconvenient. Therefore, we will now eliminate them by searching a $T(E_m)$ such that
    \begin{align}\begin{split}
        \label{equ:tmp3}
		S_1 + S_2 + S_3 &\leq S_1 + T(E_m) \left( \sum_{\ell=0}^{N_0-1} \abs{c_{m,\ell}} \right)^2\\
		&= \left( N_0^{2p} + T(E_m) \right) \left( \sum_{\ell=0}^{N_0-1} \abs{c_{m,\ell}} \right)^2.
    \end{split}\end{align}
    This is true if 
    \begin{align*}
		S_2 + S_3 - T(E_m) \left( \sum_{\ell=0}^{N_0-1} \abs{c_{m,\ell}} \right)^2 \leq 0,
	\end{align*}
	which in turn leads to 
	\begin{align*}		
        T(E_m) \geq \frac{2 N_0^p N^p}{\sum_{\ell=0}^{N_0-1} \abs{c_{m,\ell}}} E_m + \frac{N^{2p}}{\left( \sum_{\ell=0}^{N_0-1} \abs{c_{m,\ell}} \right)^2} E_m^2,
    \end{align*}
    after using the definitions of $S_2$ and $S_3$.
    Thus, for 
    \begin{align}
        \label{equ:def_tem}
    	T(E_m) \coloneqq \frac{2 N_0^p N^p}{\epsilon_m} E_m + \frac{N^{2p}}{\epsilon_m^2} E_m^2
    \end{align}
	we can bound 
    \begin{align*}
    	S_1 + S_2 + S_3 \le \left( N_0^{2p} + T(E_m) \right) \left( \sum_{\ell=0}^{N_0-1} \abs{c_{m,\ell}} \right)^2.
    \end{align*}
    Using the Cauchy-Schwarz inequality we have
    \begin{align*}
        \left( \sum_{\ell=0}^{N_0-1} \abs{c_{m,\ell}} \right) ^2 = \left( \sum_{\ell=0}^{N_0-1} \abs{c_{m,\ell}} \cdot 1 \right) ^2 \leq \sum_{\ell=0}^{N_0-1} \abs{c_{m,\ell}}^2 \cdot \sum_{\ell=0}^{N_0-1} 1^2 = N_0 \sum_{\ell=0}^{N_0-1} \diff{\abs{c_{m,\ell}}^2},
    \end{align*}
    so that with \eqref{equ:tmp1}, \eqref{equ:tmp2} and \eqref{equ:tmp3} we get
    \begin{align*}
    	\norm{(I - I_H^h I_h^H) e_h}_2^2 &\leq \tilde{C}(p)^2 \Delta x^{2p} N_0 \sum_{m=1}^M \left( N_0^{2p} + T(E_m) \right) \sum_{\ell=0}^{N_0-1} \diff{\abs{c_{m,\ell}}^2}.
    \end{align*}
    
    Further, it follows from Parseval's theorem~\cite{fourier} that
    \begin{align*}
        \sum_{\ell=0}^{N_0-1} \diff{\abs{c_{m,\ell}}^2} \leq \sum_{\ell=0}^{N-1} \diff{\abs{c_{m,\ell}}^2} = \sum_{n=1}^{N} \diff{\abs{e_{m,n}}^2}
    \end{align*}
    and thus
    \begin{align*}
    	\norm{(I - I_H^h I_h^H) e_h}_2^2 &\leq \tilde{C}(p)^2 \Delta x^{2p} N_0 \sum_{m=1}^M \left( N_0^{2p} + T(E_m) \right) \sum_{n=1}^{N} \diff{\abs{e_{m,n}}^2}\\
    	&= \tilde{C}(p)^2 \Delta x^{2p} N_0 \left( N_0^{2p} + \max_{m=1,\dots,M} T(E_m) \right) \norm{e_h}_2^2.
    \end{align*}
    Since it is the maximum norm we are interested in and not the Euclidean one, an appropriate transformation is required, given by
    \begin{align*}
        \infnorm{x} \leq \norm{x}_2 \leq \sqrt{n} \infnorm{x} \quad \forall x \in \compl^n,
    \end{align*}
    so that
    \begin{align*}
    	\infnorm{(I - I_H^h I_h^H) e_h} &\leq \norm{(I - I_H^h I_h^H) e_h}_2\\
    	&\leq \tilde{C}(p) \Delta x^{p} \sqrt{N_0 \left( N_0^{2p} + \max_{m=1,\dots,M} T(E_m) \right)} \norm{e_h}_2\\
    	&\leq  \tilde{C}(p) \Delta x^{p} \sqrt{N_0MN} \sqrt{N_0^{2p} + \max_{m=1,\dots,M} T(E_m)} \infnorm{e_h}.
    \end{align*}
    Using the triangle inequality for square roots, the sum can be split as
    \begin{align*}
    	\infnorm{(I - I_H^h I_h^H) e_h} &\leq \tilde{C}(p) \Delta x^{p} \sqrt{N_0MN} \left( N_0^{p} +  \sqrt{\max_{m=1,\dots,M} T(E_m)} \right) \infnorm{e_h}. 
    \end{align*}
    With the insertion of the definition of $T(E_m)$ presented in \fref{equ:def_tem}, we get
    \begin{align}\begin{split}
        \label{equ:first_sight}
    	\infnorm{(I - I_H^h I_h^H) e_h} &\leq \tilde{C}(p) \sqrt{N_0MN} N_0^{p} \Delta x^{p}\\
        &\quad \cdot \left( N_0^p + \sqrt{\max_{m=1,\dots,M} \frac{2 N_0^p N^p}{\epsilon_m} E_m + \frac{N^{2p}}{\epsilon_m^2} E_m^2} \right) \infnorm{e_h}.
    \end{split}\end{align}

    At the first sight, it looks like $\Delta x^p$ is the dominating term in this equation. However, a closer look reveals that the root term is in $\bigo(N^p)$ which shifts the dominance of this summand towards the remainder $E_m$. To see this, we first transform the root term by extracting $N^{2p}$:
    \begin{align*}
        \sqrt{\max_{m=1,\dots,M} \frac{2 N_0^p N^p}{\epsilon_m} E_m + \frac{N^{2p}}{\epsilon_m^2} E_m^2} &\leq 
        \sqrt{\max_{m=1,\dots,M} N^{2p} \left( \frac{2 N_0^p}{\epsilon_m} E_m + \frac{1}{\epsilon_m^2} E_m^2 \right)}\\
        &= N^p \sqrt{\max_{m=1,\dots,M} \frac{2 N_0^p}{\epsilon_m} E_m + \frac{1}{\epsilon_m^2} E_m^2}.
    \end{align*}        
    Then, we consider the definition of the step size on the fine level, namely $\Delta x_h = \frac{S}{N}$, which allows the representation of $N^p$ as $\left( \frac{S}{\Delta x_h} \right) ^p$. From this representation, it follows
    \begin{align*}
    	\Delta x^p  &\sqrt{\max_{m=1,\dots,M} \frac{2 N_0^p N^p}{\epsilon_m} E_m + \frac{N^{2p}}{\epsilon_m^2} E_m^2}\\&\quad \leq S^p \left( \frac{\Delta x_H}{\Delta x_h} \right)^{p} \sqrt{\max_{m=1,\dots,M} \frac{2 N_0^p}{\epsilon_m} E_m + \frac{1}{\epsilon_m^2} E_m^2}.
    \end{align*}
    The insertion of this estimation into \fref{equ:first_sight} finally leads to the overall result
    \begin{align*}
        &\infnorm{(I - I_H^h I_h^H) e_h} \leq \tilde{C}(p) \sqrt{N_0MN} N_0^{2p} \Delta x^{p} \infnorm{e_h}\\
        &\quad\quad\quad\quad\quad\quad + \frac{(2\pi)^p}{4p} \left( \frac{\Delta x_H}{\Delta x_h} \right)^{p} \left( \sqrt{\max_{m=1,\dots,M} \frac{2 N_0^p}{\epsilon_m} E_m + \frac{1}{\epsilon_m^2} E_m^2} \right) \infnorm{e_h},
    \end{align*}
    where we used the definition of $\tilde{C}(p)$ to eliminate $S^p$ in the second summand. 
    
    In this estimation, we can finally see that $\Delta x^p$ does not dominate the second summand anymore. It is replaced by the relation between the step size on the coarse and on the fine level which can be regarded as constant. As a result, the remainder $E_m$ is now the dominating item in the term. Hence, a formulation like
    \begin{align*}
    	\infnorm{(I - I_H^h I_h^H) e_h} &\leq C_{11} \Delta x^p \infnorm{e_h} + C_{12}(E) \infnorm{e_h}
    \end{align*}
    with $E \coloneqq (E_m)_{1 \leq m \leq M}$ is reasonable and concludes the proof.
\end{proof}

\begin{remark}
    The splitting of the sum and the consideration of the vector of remainders $E$ is needed since the constant $C_{11}$ would otherwise depend on $N^p$.
    Considering the relation $\frac{N}{S} = \frac{1}{\Delta x}$, this would mean that the approximated boundary would not depend on $\Delta x$ anymore. As a result, we would not obtain a better estimation than the simple one $\infnorm{(I-I_H^h I_h^H)(e_h)} \ \le\  C \infnorm{e_h}$, which was already used in the proof of \fref{theo:mlsdc_conv_gen}. By the applied split of the series, the term $N^p$ is replaced by $N_0^p$ which yields a more meaningful estimation as it keeps the dependence on the term $\Delta x^p$ while adding another on the smoothness of the error.
\end{remark}

\begin{remark}
    Note that ideally the error only has a few, low-frequency Fourier coefficients, i.e. $\tilde{e}_{m}(x)$ can be written as 
    \begin{align*}
        \tilde{e}_m(x) \coloneqq \frac{1}{\sqrt{N}} \sum_{\ell=0}^{N_0-1} c_{m,\ell} \exp{i \frac{2\pi}{S} \ell x}, \quad x \in [0, S], \quad m=1, \dots, M,
    \end{align*}
    using $N_0$ summands only.
    Then, $E_m = 0$ and the estimate
    \begin{align*}
    	\infnorm{(I - I_H^h I_h^H) e_h} &\leq C_{11} \Delta x^p \infnorm{e_h} + C_{12}(E) \infnorm{e_h}
    \end{align*}
    reduces to
    \begin{align*}
    	\infnorm{(I - I_H^h I_h^H) e_h} &\leq C_{11} \Delta x^p \infnorm{e_h}.
    \end{align*}
\end{remark}

The following theorem uses \fref{lem:ihh_gen} to extend \fref{theo:mlsdc_conv_gen}. In particular, the provided estimation for $\infnorm{(I - I_H^h I_h^H) (U_h - U_h^{(k)})}$ is used in the corresponding proof which results in a new convergence theorem for MLSDC.
    
\begin{theorem}
	\label{theo:mlsdc_conv_prec}
    Consider a generic initial value problem like (\ref{equ:ivp}) with a Lipschitz-continuous function $f$ on the right-hand side. Furthermore, let the conditions of \fref{lem:ihh_gen} be met.
    
    Then, if the step size $\Delta t$ is sufficiently small, MLSDC converges linearly to the solution of the collocation problem with a convergence factor in $\bigo((\Delta x^p + C(E)) \Delta t + \Delta t^2)$, i.e. the following estimate for the error is valid:
    \begin{align}
   	    \label{equ:mlsdc_conv_prec_iter}
        \infnorm{U_h - U_h^{(k)}} \leq ((C_{13} \Delta x^p + C_{14}(E)) \Delta t + C_{15} \Delta t^2) \infnorm{U_h - U_h^{(k-1)}},
    \end{align}
    where $\Delta x \equiv \Delta x_H$ is defined as the resolution in space on the coarse level $\Omega_H$ of MLSDC and the constants $C_{13}$, $C_{14}(E)$ and $C_{15}$ are independent of $\Delta t$.
   	
   	If, additionally, the solution of the initial value problem $u$ is $(M+1)$-times continuously differentiable, the LTE of MLSDC compared to the solution of the ODE can be bounded by:
   	\begin{align}\begin{split}
   	    \label{equ:mlsdc_conv_prec_lte}
		\infnorm{\bar{U_h} - U_h^{(k)}} &\leq C_{17} \Delta t^{M+1} \norm{u}_{M+1}\\
		&\quad+ \sum_{l=0}^{k} C_{l+18} (\Delta x^p + C_{14}(E))^{k-l} \Delta t^{k_0+k+l} \norm{u}_{k_0+1}
	\end{split}\end{align}
	where the constants $C_{17}, \dots, C_{k+18}$ are independent of $\Delta t$, $k_0$ denotes the approximation order of the initial guess $U_h^{(0)}$ and $\norm{u}_p$ is defined by $\infnorm{u^{(p)}}$.
\end{theorem}
\begin{proof}
	The proof is similar to the one of \fref{theo:mlsdc_conv_gen} but differs in the used estimation for $\infnorm{(I - I_H^h I_h^H) (U_h - U_h^{(k)})}$. Here, \fref{lem:ihh_gen} instead of the simple norm compatibility inequality is used for this purpose. Based on the estimations (\ref{equ:u-1_rek_u-half}), (\ref{equ:u-half_Ihh}) and (\ref{equ:u-half_H}), arising in the proof of the mentioned theorem, it follows
	\begin{align*}
    	\infnorm{U_h - U_h^{(k+1)}} \leq \tilde{C}_1 \Delta t \infnorm{(I - I_H^h I_h^H) (U_h - U_h^{(k)})} + C_{15} \Delta t^2 \infnorm{U_h - U_h^{(k)}}.
	\end{align*}
    As already mentioned, we will now apply \fref{lem:ihh_gen}, namely
    \begin{align*}
        \infnorm{(I - I_H^h I_h^H)(U_h - U_h^{(k)})} \leq (C_{11} \Delta x^p + C_{12}(E)) \infnorm{U_h - U_h^{(k)}},
    \end{align*}
    which yields
    \begin{align*}
    	\infnorm{U_h - U_h^{(k+1)}} &\leq (C_{13} \Delta x^p + C_{14}(E)) \Delta t \infnorm{U_h - U_h^{(k)}}\\
    	&\quad + C_{15} \Delta t^2 \infnorm{U_h - U_h^{(k)}}\\
    	&= ((C_{13} \Delta x^p + C_{14}(E)) \Delta t + C_{15} \Delta t^2) \infnorm{U_h - U_h^{k}}.
    \end{align*}
    This concludes the proof of \fref{equ:mlsdc_conv_prec_iter}.
    
	The proof of \fref{equ:mlsdc_conv_prec_lte} is again similar to the one of the second equation in \fref{theo:sdc_conv}, using the previous result. Additionally, the binomial theorem is applied to simplify the arising term $((C_{13} \Delta x^p + C_{14}(E)) \Delta t + C_{15} \Delta t^2)^k$.
\end{proof}

\begin{remark}
    It is here that a possible dependency of $f$'s Lipschitz constant on $\Delta x$ plays a key role.
    Similar to the observations before, we find that in this case \fref{equ:mlsdc_conv_prec_iter} needs to be replaced with
    \begin{align*}
        \infnorm{U_h - U_h^{(k)}} \leq ((C_{13} \Delta x^{p} + C_{14}&(E))C(\delta^{-1}) \Delta t \\&+ C_{15}(\delta^{-2}) \Delta t^2) \infnorm{U_h - U_h^{(k-1)}},
    \end{align*}
    where $\delta$ denotes the difference between spatial and temporal resolution (up to constants) and comes from the initial step size restriction of SDC.
    The term $C_{13} \Delta x^{p} + C_{14}(E)$ itself does not depend on $\delta$, since it comes from the remainder of the interpolation estimate in lemma~\ref{lem:ihh_gen}, where $f$ and therefore its Lipschitz constant as well as the step size does not play a role.
    As before, \fref{equ:mlsdc_conv_prec_lte} has to be modified, now including the term $\delta^{-(k-l+1)}$ in the sum. 
    The constant $C_{17}$ is still independent of $\delta$.
\end{remark}

\begin{remark}
	Similar to \fref{cor:conv_last}, it can be proven that the order limit $M+1$ in \fref{theo:mlsdc_conv_prec} can be replaced by $2M$ if only the error at the last collocation node is considered.
\end{remark}

The theorem states that, under the named conditions, MLSDC converges linearly with a convergence rate of $\bigo((\Delta x^p + C(E)) \Delta t + \Delta t^2)$ to the collocation solution if $(C_{13} \Delta x^p + C_{14}(E)) \Delta t + C_{15} \Delta t^2 < 1$. This means that if $\Delta x^p$ and the vector of remainders $E$ are sufficiently small, the error of MLSDC decreases by two orders of $\Delta t$ with each iteration, which indeed represents an improved convergence behavior compared to the one described in \fref{theo:mlsdc_conv_gen}. Otherwise, i.e.\ if $\Delta x^p$ and $E$ are not that small, it only decreases by one order in $\Delta t$ which is equivalent to the result of the previous theorem. 

In the second equation of the theorem, it can be seen that again, $\Delta x^p$ and $E$ are the crucial factors here. If they are small enough such that $\Delta t^{k_0 + 2k}$ is the leading order, MLSDC converges with order $\min(k_0+2k-1, 2M-1)$ and thus gains two orders per iteration. Otherwise, the convergence order is only $\min(k_0+k-1, 2M-1)$, i.e.\ the error decreases by one order in $\Delta t$ in each iteration. 

Note that, as a result, it is advisable to use a high interpolation order $p$ and a small spatial step size $\dx$ on the coarse level in practical applications of MLSDC. This theoretical result matches the numerical observations described in \cite{mlsdc-2}. In section 2.2.5 of this paper, it is mentioned that the convergence properties of MLSDC seem to be highly dependent on the used interpolation order and resolution in space. Moreover, it was said that in the considered numerical examples a high resolution in space, i.e.\ a small $\Delta x$, led to a lower sensitivity on the interpolation order $p$. Our theoretical investigation provides an explanation for this behavior.

It seems reasonable to use a similar approach to determine the conditions for a higher convergence order of MLSDC if coarsening in time instead of space is used. Analogous to \fref{equ:interpol_error} in the proof of \fref{lem:ihh_gen}, the Lagrangian error estimation could be used for this purpose, resulting in the following estimation
\begin{align}
    \label{equ:err_coarse_time}
    \infnorm{(I- I_H^h I_h^H)(U_h - U_H^{(k)})} \leq \frac{\Delta \tau^p}{4p} \infnorm{e^{(p)}(t)},
\end{align}
where $I_H^h$ and $I_h^H$ denote temporal transfer operators now and $e(t)$ is defined as the continuous error of MLSDC compared to the collocation solution. In this case, however, the function $e(t)$ is implicitly known and thus does not have to be approximated by an iDFT. In particular, it is a polynomial of degree $M \equiv M_h$ as both the collocation solution $U$ and each iterate $U_h^{(k)}$ of MLSDC are polynomials of that degree, respectively. This can be seen by considering that $\dt Q F(U)$ as well as $\dt Q_\Delta F(U)$ essentially represent a sum of integrals of Lagrange polynomials which apparently results in a polynomial. Consequently, the $p$-th derivative of $e(t)$ is a polynomial of degree $M_h-p$. The maximal interpolation order $p$ is the number of collocation nodes $M_H$ on the coarse level. Here, we will assume that $p = M_H$, i.e. the maximal interpolation order is used. Note that $e^{(p)}(t) = 0$ for $p > M_h$ and hence $\infnorm{(I- I_H^h I_h^H)(U_h - U_H^{(k)})} = 0$ for $M_H = M_h$ which is consistent with the expected behavior as it means that no coarsening is used at all.
As a conclusion, it can be said that, according to \fref{equ:err_coarse_time} the improved convergence behavior of MLSDC using coarsening in time is dependent on the used time step size $\Delta \tau = C \dt$ and the number of collocation nodes on the coarse level $M_H$. Note that it is also dependent on the specific coefficients of $e^{(p)}(t)$. However, these are highly dependent on the right-hand side $f$ of the IVP and thus cannot be controlled by any method parameters.

In summary, two convergence theorems for MLSDC were established in this section. While the first one, \fref{theo:mlsdc_conv_gen}, represents a general statement on the convergence of the method, the second one, \fref{theo:mlsdc_conv_prec}, provides theoretically established guidelines for the parameter choice in practical applications of MLSDC in order to achieve an improved convergence behavior of the method. In the next section, we will examine numerical examples of MLSDC to check if the resulting errors match those theoretical predictions. 

\section{Numerical Results}
\label{sec:examples}


In this section, the convergence behavior of MLSDC, theoretically analyzed in the previous section, is verified by numerical examples. The method is applied to three different initial value problems and the results are compared to those from classical, single-level SDC.
The key question here is whether the conditions derived in the previous sections (smoothness, high spatial/temporal resolution and high interpolation order) are actually sharp, i.e.\ whether MLSDC does indeed show only low order convergence if any of these conditions are violated. 
The corresponding programs were written in Python using the \texttt{pySDC} code \cite{pysdc,10.1145/3310410}.




\subsection{Heat equation}

The first numerical example is the one-dimensional heat equation defined by the following initial value problem:
\begin{equation}\begin{gathered}
    \label{equ:heat_ode}
    \frac{\partial}{\partial t} u(x,t) = \nu \frac{\partial^2}{\partial x^2} u(x,t), \quad \forall t \in [0,t_{end}], x \in [0,1]\\
    u(0,t) = 0, \quad u(1,t) = 0,\\
    u(x,0) = u_0(x),
\end{gathered}\end{equation}
where $u(x,t)$ represents the temperature at the location $x$ and time $t$ and $\nu > 0$ defines the thermal diffusivity of the medium.
This partial differential equation is discretized in space using standard second-order finite differences with $N$ degrees-of-freedom.


As initial value a sine wave with frequency $\kappa$ is selected, i.e.\ $u_0(x) = \sin(\kappa \pi x)$. Under these conditions, the analytical solution of the spatially discretized initial value problem is given by
\begin{align*}
	\vec{u}(t) = \sin(\kappa \pi \vec{x})e^{-t \nu \rho}\text{ with } \rho = \frac{1}{\Delta x^2} (2 - 2\cos(\pi \nu \vec{x}))
\end{align*}
with $\vec{x} \coloneqq (x_n)_{1 \leq n \leq N}$ and an element-wise application of the trigonometric functions.
For the tests, we choose $\kappa=4$, $\nu = 0.1$ and $M=5$ Gauß-Radau collocation nodes.

Note that we are using this linear ODE for our tests even though the linearity of the right-hand side $f$ is not a necessary condition in theorems \ref{theo:sdc_conv}, \ref{theo:mlsdc_conv_gen} or \ref{theo:mlsdc_conv_prec}. In fact, $f$ is just assumed to be Lipschitz continuous. However, we will consider the heat equation here since it is well studied and has a convenient exact solution needed to compute the errors of SDC and MLSDC. 

The following tests are structured in a particular way: In the first one, we will adjust the method parameters according to the results of \fref{theo:mlsdc_conv_prec} to observe an improved convergence of MLSDC over SDC. More specifically, we will use a small spatial step size $\Delta x$, a high interpolation order $p$ and try to generate smooth errors using smooth initial guesses for the iteration. In a second step, we will then subsequently change these parameters leading to a lower convergence order of MLSDC as described in \fref{theo:mlsdc_conv_gen}. Thereby, we will reveal the dependence of MLSDC's convergence behavior on those parameters and simultaneously verify the general minimal achievable convergence order of the method. Altogether, this will confirm the theoretical results of the previous section.



For the first test, the number of degrees-of-freedom was set to $N_h = 255$ on the fine and $N_H = 127$ on the coarse level of MLSDC. This parameter particularly determines the spatial grid size $\Delta x = \frac{1}{N+1}$.
As discussed before, we use injection as restriction and a piecewise $p$-th order Lagrange interpolation as interpolation, for now with $p=8$.
Moreover, the smooth initial value $u_0(x) = \sin(4\pi x)$ was spread across the different nodes $\tau_m$ to form the initial guess.


\begin{figure}
    \includegraphics[width=\textwidth]{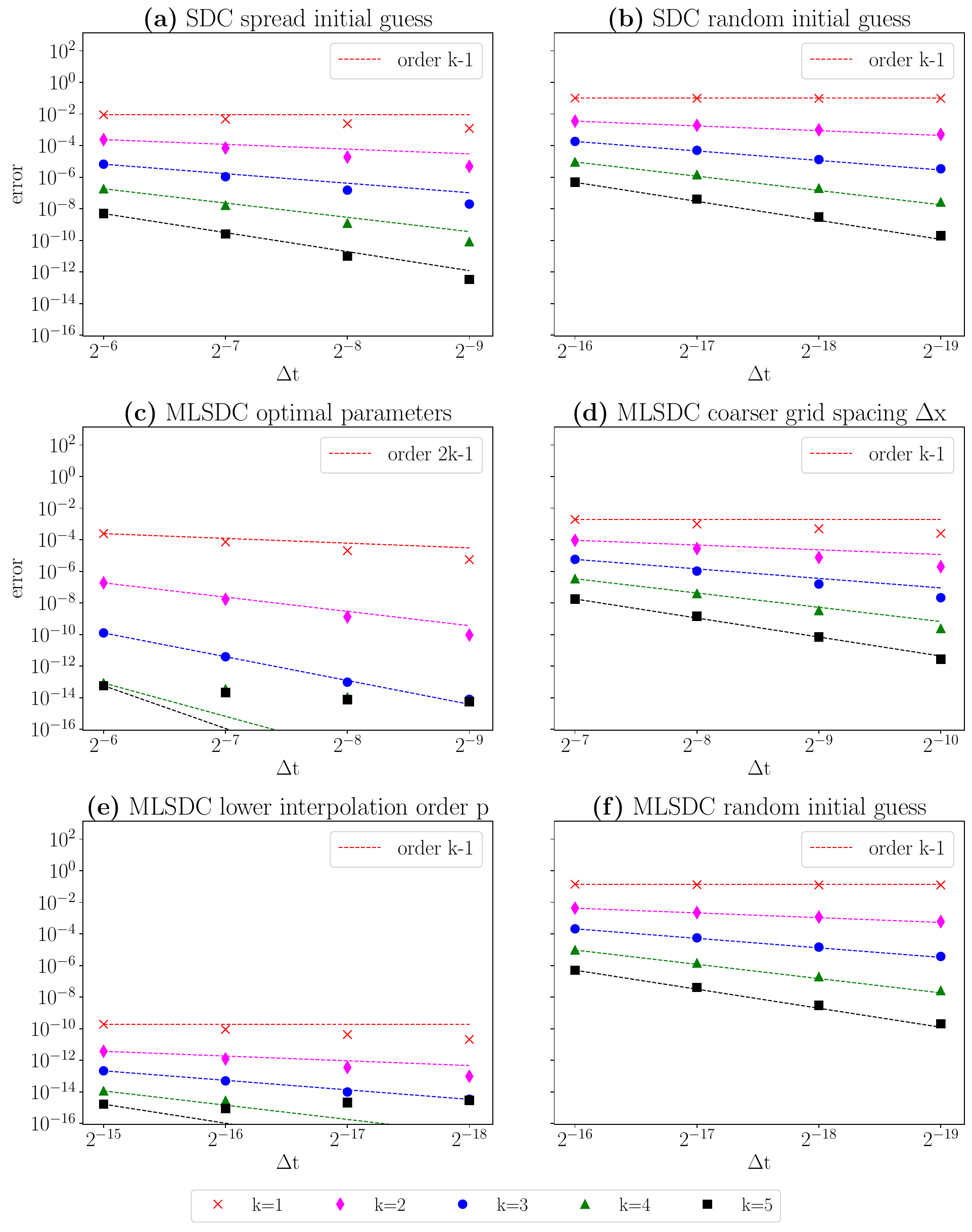}
    \caption{Convergence behavior of SDC and MLSDC applied to the discretized one-dimensional heat equation with coarsening in space (for MLSDC) and different parameters}
    \label{fig:heat}
\end{figure}

An illustration of the corresponding numerical results is shown in \fref{fig:heat}c, with the reference SDC result in \fref{fig:heat}a. MLSDC was applied with different step sizes $\Delta t$ and numbers of iterations $k$ to the considered problem and the resulting errors were plotted as points in the respective graphs. The drawn lines, on the other hand, represent the expected behavior, i.e.\ the predicted convergence orders of the method according to \fref{theo:mlsdc_conv_prec}. In particular, we assume that the terms $\Delta x^p$ and $C(E)$ are sufficiently small such that $\Delta t^{k_0+2k}$ is the leading order in the corresponding error estimation. As a result, MLSDC is expected to gain two orders per iteration. In the figure, it can be seen that nearly all of the computed points lie on the expected lines (except when the errors hit the precision limit at approx. $10^{-14}$) which always start at the error resulting for the largest step size. Therefore, the numerical results match the theoretical predictions. 

Note that the gain in accuracy when comparing two runs with $k$ and $k+1$ for a fixed $\dt$ is not has high as one might expect. As mentioned in Remark~\ref{rem:Lip_k}, this is due to the fact that the error is multiplied with $L^k$, where $L>1$ is the Lipschitz constant of the discretized right-hand side of the ODE~\eqref{equ:heat_ode}.


If, by contrast, the spatial grid size $\Delta x$ is chosen to be significantly larger, in particular as large as $\frac{1}{16}$ on the fine and $\frac{1}{8}$ on the coarse level, the leading order in \fref{theo:mlsdc_conv_prec}, presenting an error estimation for MLSDC, changes to $\Delta t^{k_0+k}$. Hence, in this example, we expect MLSDC to only gain one order in $\Delta t$ with each iteration, as SDC does and as it was described in the general convergence \fref{theo:mlsdc_conv_gen}. The corresponding numerical results, presented in \fref{fig:heat}d, confirm this prediction. For both methods, the error decreases by one order in $\Delta t$ with each iteration.

Another possible modification of the first example is a decrease of the interpolation order $p$. \Fref{fig:heat}e shows the numerical results if this parameter is changed to $p=4$. Apparently, this also leads to an order reduction of MLSDC compared to \fref{fig:heat}c. According to \fref{theo:mlsdc_conv_prec}, this is a reasonable behavior. In particular, the leading order in the presented error estimation is again reduced to $\Delta t^{k_0+k}$ due to the higher magnitude of $\Delta x^p$. Besides, it should be noted that the considered values of $\Delta t$ are significantly smaller here. This is caused by the fact that MLSDC does not converge for greater values of this parameter, i.e.\ the upper bound for $\Delta t$, implicitly occurring in the assumptions of the respective theorem, seems to be lower here. The smaller step sizes $\Delta t$ also entail overall smaller errors. As a result, the accuracy of the collocation solution is reached earlier which explains the outliers in the considered plots.

The third necessary condition for the improved convergence of MLSDC is the magnitude of the remainders $E_m$ or, in other words, the smoothness of the error. In this context, we will now have a look at the changes which result from a higher oscillatory initial guess. In particular, we will assign random values to $U^{(0)}$. The corresponding errors are shown in \fref{fig:heat}b for SDC and \fref{fig:heat}f for MLSDC. It can be seen that this change results again in a lower convergence order of MLSDC, in particular it gains one order per iteration as SDC. Since this time, as the crucial term $\Delta x^p$ is left unchanged, the result can only be assigned to a higher value of $C(E)$ and thus to an insufficient smoothness of the error. This may lead to the assumption that for this problem type a smooth initial guess is a sufficient condition for the smoothness of the error and thus, a low value of $C(E)$.

\subsection{Allen-Cahn equation}

The second test case is the non-linear, two-dimensional Allen-Cahn equation
\begin{align}\label{eq:ac}
	u_t &= \Delta u + \frac{1}{\epsilon^2} u(1-u^2)\quad \mathrm{on}\quad [-0.5, 0.5]^2\times[0,T],\ T>0,\\
    u(x,0) &= u_0(x),\quad x\in[-0.5,0.5]^2,\nonumber
\end{align}
with periodic boundary conditions and scaling parameter $\epsilon > 0$.
We use again second-order finite differences in space and choose a sine wave in 2D as initial condition, i.e.\ $u_0(x) = \sin(\kappa \pi x)\sin(\kappa\pi y)$.
There is no analytical solution, neither for the continuous nor for the spatially discretized equations. 
Therefore, reported errors are computed against a numerically computed high-order reference solution.
For the tests, we choose $\kappa=4$, $\epsilon = 0.2$ and $M=3$ Gauß-Radau collocation nodes.

The tests are structured precisely as for the heat equation: we first show second-order convergence factors using appropriate parameters and then test the sharpness of the conditions on smoothness, the resolution and the interpolation order.
For the first test, the number of degrees-of-freedom per dimension was set to $N_h = 128$ on the fine and $N_H = 64$ on the coarse level of MLSDC.
Transfer operators are the same as before.

\begin{figure}
    \includegraphics[width=\textwidth]{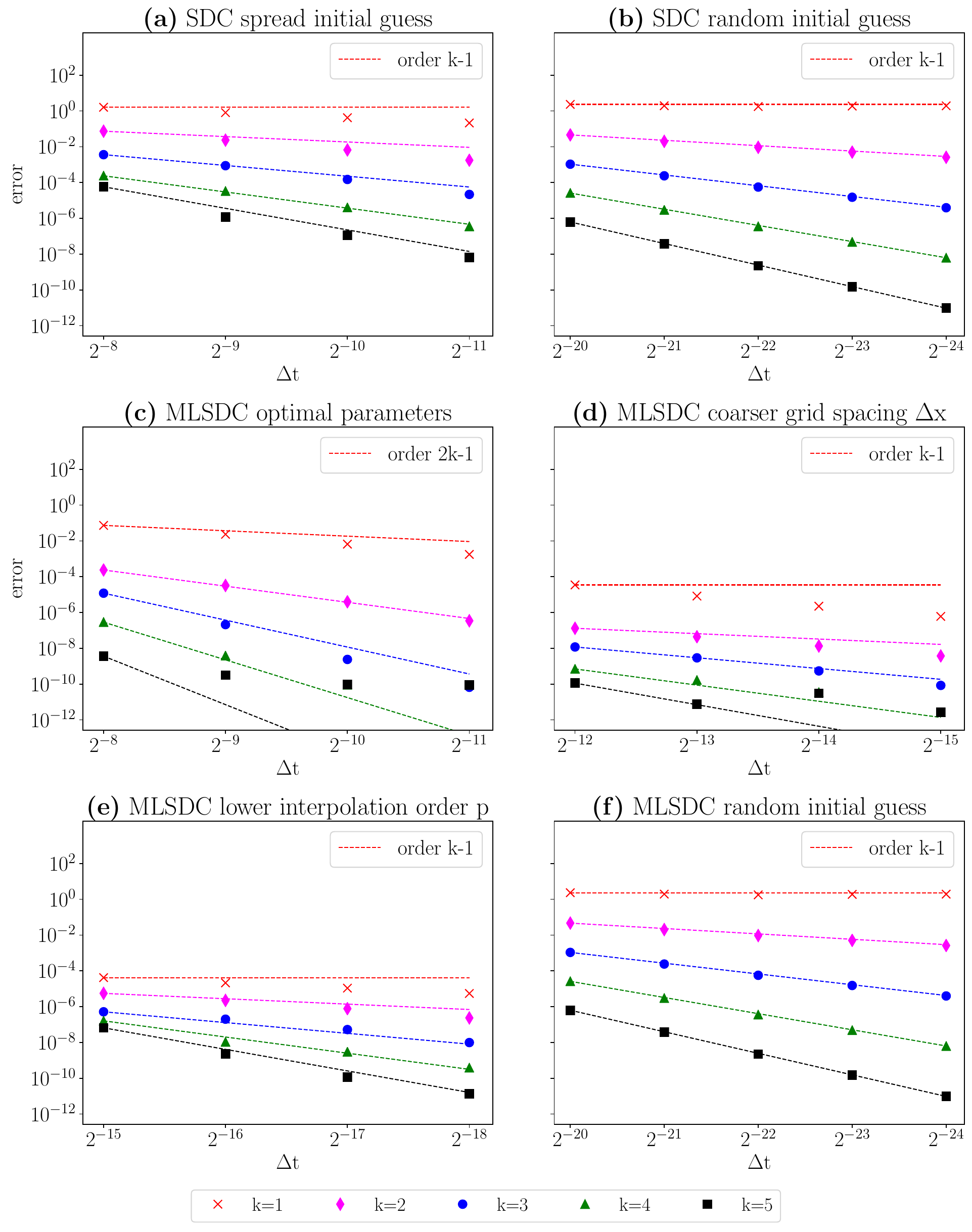}
    \caption{Convergence behavior of SDC and MLSDC applied to the discretized two-dimensional Allen-Cahn equation with coarsening in space (for MLSDC) and different parameters}
    \label{fig:allencahn}
\end{figure}

Figure~\ref{fig:allencahn} shows the results of our tests for the Allen-Cahn equation for both SDC and MLSDC.
The main conclusion here is the same as before: using less degrees-of-freedom (here $N=32$ on the fine level instead of $128$), a lower interpolation order (here $p=2$ instead of $8$) or a non-smooth (here random) initial guess leads to a degraded order of the convergence factor.

\subsection{Auzinger's test case}

The third test case is the following two-dimensional ODE introduced in \cite{auzinger2004modified}:

\begin{equation}\begin{gathered}
    \dot{u} = \begin{pmatrix} \dot{x} \\ \dot{y} \end{pmatrix} 
    = \begin{pmatrix} -y - \lambda x(1-x^2-y^2) \\ x - \lambda \rho y(1-x^2-y^2) \end{pmatrix}, \quad \forall t \in [0,t_{end}] \\\
    u(0) = u_0,
\end{gathered}\end{equation}
where $\lambda<0$ determines the stiffness of the problem and $\rho>0$ is a positive parameter.
For the tests, we choose $\lambda=-0.75$, $\rho=3$ and $u_0 = (1,0)^T$.
The analytical solution of this initial value problem is known. It is given by
\begin{align*}
    u(t) = \begin{pmatrix} x(t) \\ y(t) \end{pmatrix} = \begin{pmatrix} \cos(t) \\ \sin(t) \end{pmatrix},\quad t \in [0,t_{end}].
\end{align*}

The corresponding tests are structured in a similar way as before but this time the ODE version of \fref{theo:mlsdc_conv_prec}, given by \fref{equ:err_coarse_time}, is considered. So, first appropriate parameters are used to reach second-order convergence of MLSDC, and then the sharpness of the implied conditions is tested. In particular, the improved convergence behavior of MLSDC is expected to depend on the time step size $\Delta \tau = C \Delta t$, the (now temporal) interpolation order $p$ and the smoothness of the error in time. In our tests, we always used the maximal interpolation order $p=M_H$ corresponding to the number of collocation nodes on the coarse level, since otherwise it was not possible to get a second order convergence at all. The number of nodes on the fine grid was chosen to be $M_h=8$.

\begin{figure}
    \includegraphics[width=\textwidth]{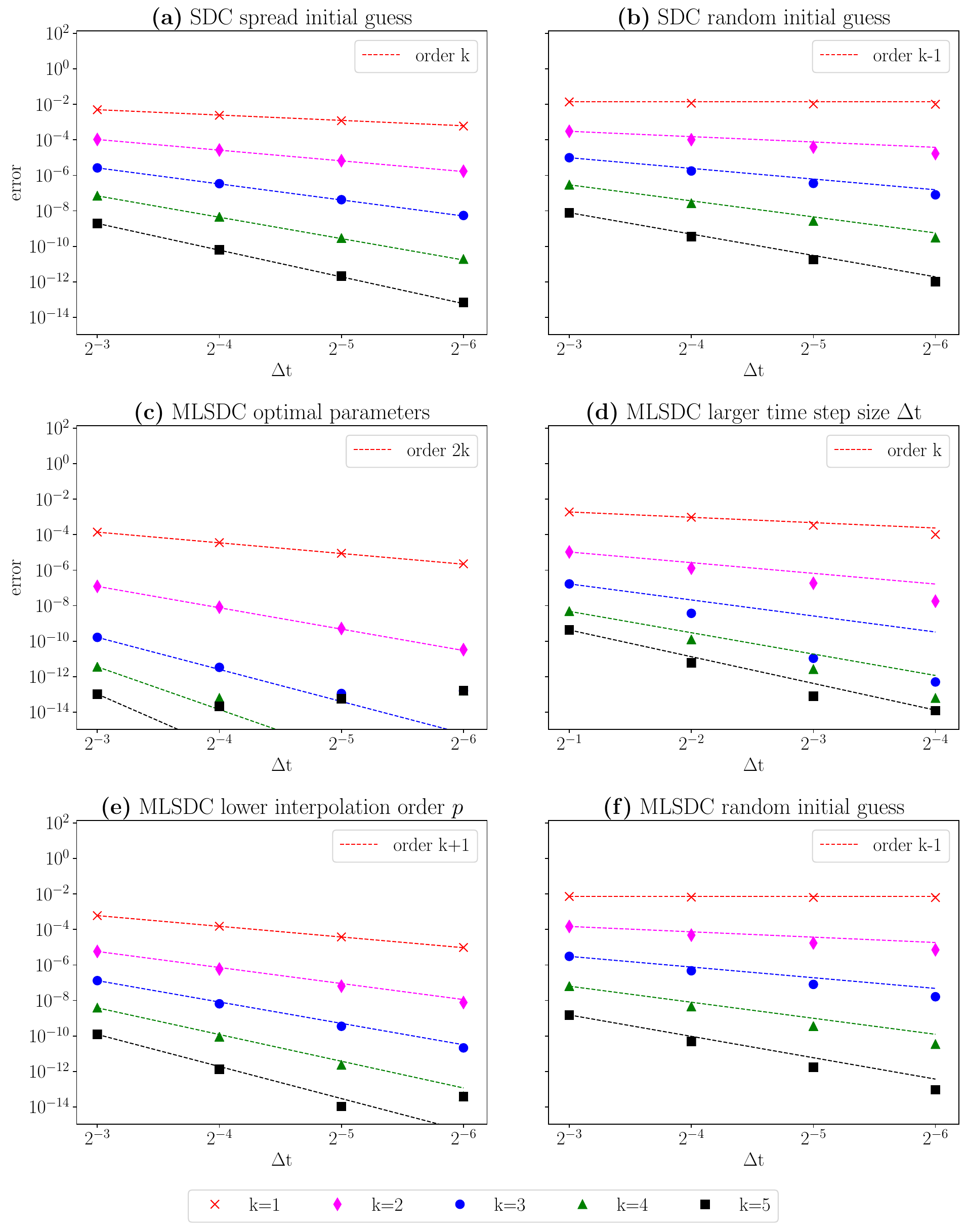}
    \caption{Convergence behavior of SDC and MLSDC applied to the Auzinger problem with coarsening in the collocation nodes (for MLSDC) and different parameters}
    \label{fig:auzinger}
\end{figure}

The numerical results are shown in \fref{fig:auzinger}. Again, they agree with our theoretical predictions: All of the three conditions implied by \fref{equ:err_coarse_time} need to be fulfilled to reach second-order convergence of MLSDC. A larger time step size (here $\Delta t \in [2^{-1}, 2^{-4}]$ instead of $[2^{-3}, 2^{-6}]$), a lower interpolation order (here $p=M_H=2$ instead of $6$) or a non-smooth (here random) initial guess immediately led to a decrease in the order.

However, there are a few oddities in the graphs that we would like to discuss here. First of all, the orders shown in \fref{fig:auzinger}a, c and d are not $2k-1$ and $k-1$ as we would expect, but rather $2k$ and $k$. This behavior is probably related to the $k_0$-term in the estimates which stems from the initial guess of SDC and MLSDC. This explanation would also agree with the result that this additional order gets lost if a random initial guess is used (see \fref{fig:auzinger}b, f).
Aside from that, it should be noted that the use of a lower interpolation order (\fref{fig:auzinger}e) led to a convergence order of $k+1$ instead of $k$ as we would have expected. The reason for this is not clear but could be related to \fref{equ:mlsdc_conv_prec_lte} which implies that all orders between $k$ and $2k$ can potentially be reached. In any case, the result shows that the second-order convergence is lost if the interpolation order is decreased. Finally, we want to discuss the plot in \fref{fig:auzinger}d resulting from the use of a larger time step size $\Delta t$.
It can be seen that the data points do not perfectly agree with the predicted lines here. Apparently, the numerical results are often much better than expected. However, they do not reach order $2k$ and hence confirm our theory that the second-order convergence of MLSDC is also dependent of a small time step size. The deviations in the data are, in fact, not too surprising here, considering that the time step size is a very crucial parameter for the convergence of MLSDC in general. As described in \fref{theo:mlsdc_conv_gen} and \ref{theo:mlsdc_conv_prec}, $\Delta t$ has to be small enough in order for MLSDC to converge at all. For that reason, the possible testing scope for the time step size is rather small, making it difficult to find appropriate parameters where MLSDC converges exactly with order $k$.

\subsection{Further observations}\label{sec:examples_disc}

The artifacts described above shed some light on the ``robustness'' of the results, a fact that we would like to share here: during the tests with all three examples, we saw that it is actually very hard to get these more or less consistent results. 
All model and method parameters had to be chosen carefully in order to support the theory derived above so clearly.
In many cases the results were much more inconsistent, showing e.g.\ convergence orders somewhere between $k$ and $2k$, changing convergence orders or stagnating results close to machine precision or discretization errors.
None of the tests we did contradicted our theoretical results, though, but they revealed that the bounds we obtained are indeed rather pessimistic.

\diff{
\begin{table}[t]
\begin{tabular}{|c|c|ccc|}
	\hline
	
	\textbf{Example}	& \textbf{Method}	& \multicolumn{3}{c|}{\textbf{Order in $\Delta t$}} \\ 
	
	\hline\hline
	
	\multirow{4}{*}{Heat1D} & & 
	$2^{-6} \rightarrow 2^{-7}$ & $2^{-7} \rightarrow 2^{-8}$ & $2^{-8} \rightarrow 2^{-9}$ \\
	\hline
	&	SDC&    			0.844&     0.928&	0.969   \\
	&	MLSDC&  			0.196&    -0.442&   -1.064  \\
	&	MLSDC ($k=1,2$)&	1.632&     1.754&   0.443   \\
	
	\hline\hline
	
	\multirow{4}{*}{Allen-Cahn} & &
	$2^{-8} \rightarrow 2^{-9}$ & $2^{-9} \rightarrow 2^{-10}$ & $2^{-10} \rightarrow 2^{-11}$ \\
	\hline
	&	SDC&    			1.537&     0.574&     0.749     \\
	&	MLSDC&  			-0.247&    -3.341&    -1.097    \\
	&	MLSDC ($k=1,2$)&  	2.7652&    2.719&     1.629     \\
	
	\hline\hline
	
	\multirow{4}{*}{Auzinger} & &
	$2^{-3} \rightarrow 2^{-4}$ & $2^{-4} \rightarrow 2^{-5}$ & $2^{-5} \rightarrow 2^{-6}$ \\
	\hline
	&	SDC&    			0.968&     0.988&		0.976    \\
	&	MLSDC&		   		-2.762&    -2.062&     	-0.413   \\
	&	MLSDC ($k=1,2$)&	1.799&     1.202&       -4.125   \\
	
	\hline
	
\end{tabular}
\caption{\diff{Order of convergence in $\dt$, computing the error at iteration $k$ vs.~the error at iteration $k+1$, taking the mean over these ratios for a fixed $\dt$ and then computing the order of convergence when going from one $\dt$ to a finer one $\dt/2$ (which is indicated by the ``$\rightarrow$''). Ideally, SDC shows order $1$, MLSDC order $2$. The third row for each example only considers the first two iterations to avoid noise from stalling convergence.}}
\label{tab: conv_order_dt}
\end{table}

In Table~\ref{tab: conv_order_dt} we show the order of convergence in $\dt$ for the three test problems.
More precisely, we compute for a fixed $\dt$ the ratio between the error at iteration $k$ and the error at iteration $k+1$.
This corresponds, in a sense, to the convergence order in $k$ and in order to check whether this is indeed of order $\dt$ for SDC (i.e.~gaining one order of accuracy per iteration) and of order $\dt^2$ for MLSDC (i.e.~gaining two orders of accuracy per iteration), we take the mean over the ratios for a fixed $\dt$ and compute the order by comparing two different $\dt$.
The results can be seen in Table~\ref{tab: conv_order_dt} and they clearly indicate how noisy the convergence results are. 
While for SDC we do indeed see orders around $1$, the results for MLSDC are far away from clear or consistent.
This is due to stalling convergence and outliers, as we have already seen in the plots above.
When considering only the first two iterations (third row in each example), the results get closer to order $2$.
Again, this shows that obtaining consistent results is a rather delicate task.
}

In addition, one may wonder why the time step sizes are chosen so small in many of the tests above, especially since we deal with implicit time stepping schemes.
In Figure~\ref{fig:bigdt} we show the results for the same three equations as before, but now with larger $\dt$.
On the left, SDC is shown, while the right column shows MLSDC with otherwise ideal parameter choices.
The measured errors still follow the lines indicating the expected orders, but the results are less consistent.
We see e.g.~outliers in~\ref{fig:bigdt}(a) and (b), early stagnation because of large $\dt$ in (d) as well as a somewhat unclear order in (f).
Still, even for these choices of $\dt$ both SDC and MLSDC converge reasonably well and fast. 
However, this is beyond the range of the convergence proofs presented here, so that convergence, convergence bounds and orders of accuracy cannot be guaranteed.
For applying the convergence theorems both of SDC and MLSDC, certain bounds on $\dt$ must be taken into account, although actual results suggest that these bounds are way too restrictive.
These bounds are present for most SDC convergence results (see e.g.~\cite{hagstrom,huang,causley}), at least those relying on the matrix formulation we used as our starting point.
Removing or at least relaxing these bounds in this approach is a promising further research direction for both SDC and MLSDC. 

\begin{figure}
    \includegraphics[width=\textwidth]{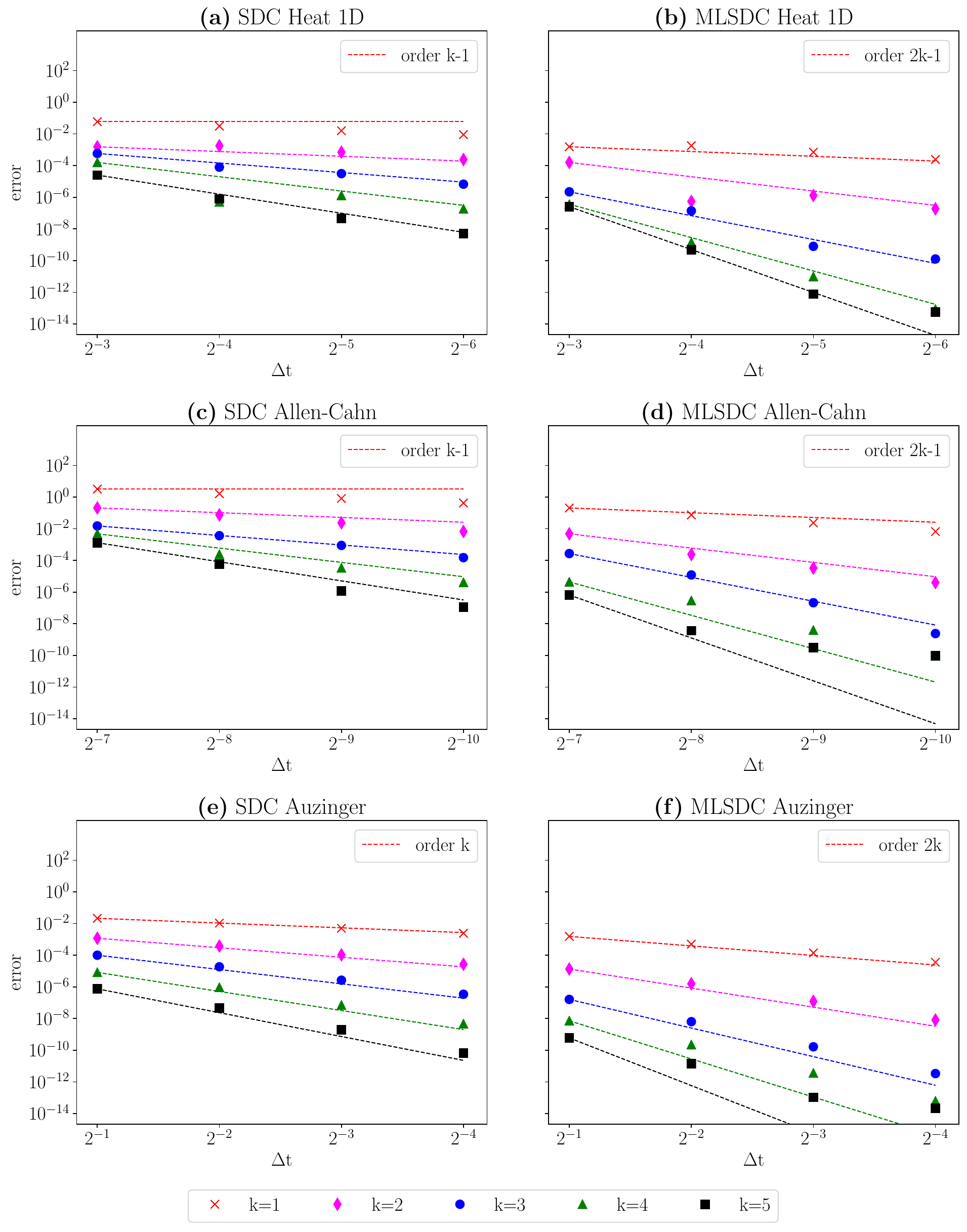}
    \caption{Convergence behavior of SDC and MLSDC for all three examples, now with larger $\dt$.}
    \label{fig:bigdt}
\end{figure}

\section{Conclusions and Outlook}


In this paper, we established two convergence theorems for multi-level spectral deferred correction (MLSDC) methods, using similar concepts and ideas as those presented in \cite{tang} for the proof of the convergence of SDC. In the first theorem, namely \fref{theo:mlsdc_conv_gen}, it was shown that with each iteration of MLSDC the error compared to the solution of the initial value problem decreases by at least one order of the chosen step size $\Delta t$, limited by the accuracy of the underlying collocation solution. The corresponding theorem only requires the operator on the right-hand side of the considered initial value problem to be Lipschitz-continuous, not necessarily linear, and the chosen time step size $\Delta t$ to be sufficiently small.
Consequently, we found a first theoretical convergence result for MLSDC proving that it converges as good as SDC does. However, we would expect and numerical results already indicated that the additional computations on the coarse level, more specifically the SDC iterations performed there, would lead to an improved convergence behavior of the method. 

For that reason, we analyzed the errors in greater detail, leading to a second theorem on the convergence of MLSDC, namely \fref{theo:mlsdc_conv_prec}. Here, we focused on a specific coarsening strategy and transfer operators. In particular, we considered MLSDC using coarsening in space with Lagrangian interpolation. Given these assumptions, we could prove that, if particular conditions are met, the method can even gain two orders of $\Delta t$ in each iteration until the accuracy of the collocation problem is reached. This consequently led us to theoretically established guidelines for the parameter choice in practical applications of MLSDC in order to achieve the described improved convergence behavior of the method. More specifically, the corresponding theorem says that for this purpose the spatial grid size on the coarse level has to be small, the interpolation order has to be high and the errors have to be smooth. 
We presented numerical examples which confirm these theoretical results. In particular, it could be observed that the change of one of those crucial parameters immediately led to a decrease in the order of accuracy. Essentially, it resulted in a convergence behavior as it was described in the first presented theorem.

Besides the research direction mentioned in Sect.~\ref{sec:examples_disc}, there are several open questions related to the presented work which have not yet been investigated. Three of them are briefly discussed here.

\textbf{More information, better results.}
The results presented here are quite generic.
As a consequence, since we only assume Lipschitz continuity of the right-hand side of the ODE and do not pose conditions on the SDC preconditioner, both constants and step size restrictions are rather pessimistic.
Using more knowledge of the right-hand side or the matrix $Q_\Delta$ will yield better results, as it already did for SDC.
Since the goal of this paper is to establish a baseline for convergence of MLSDC, exploiting this direction, especially with respect to the treatment of convergence in the stiff limit as done in~\cite{sdc-lu} for SDC, is left for future work.

\textbf{Smoothness of the error.}
The second theorem, describing conditions for an improved convergence behavior of MLSDC, has a drawback regarding its practical significance.  
The way \fref{theo:mlsdc_conv_prec} is currently proven requires a smooth error after its periodic extension.
This occurring condition of a smooth error does not always apply and is in particular not easy to control. 
Essentially, something like a smoothing property would be needed to ensure that the error always becomes smooth after enough iterations. Numerical results indicate that this property apparently does not hold for SDC, though~\cite{conv-pfasst}. In this context, however, it would be sufficient if we could at least control this condition, i.e.\ derive particular criteria for the parameters of the method ensuring the errors to be smooth. The numerical examples presented in \fref{sec:examples} particularly lead to the assumption that the selection of a smooth initial guess $U^{(0)}$ would result in smooth errors for $U^{(k)}$, $k \geq 1$, at least for a particular set of problems. 

\textbf{Other extensions of SDC.}
Furthermore, it could be tried to adapt the presented convergence proofs of MLSDC to other extensions and variations of SDC, as for example the parallel-in-time method PFASST (Parallel Full Approximation Scheme in Space and Time)~\cite{pfasst} or general semi-implicit and multi-implicit formulations of SDC (SISDC/MISDC)~\cite{imex-1,imex-2}. Whereas an adaptation to SISDC and MISDC methods seems to be rather straightforward \cite{causley}, we found that the application of similar concepts and ideas to prove the convergence of PFASST may involve some difficulties. In particular, the coupling of the different time steps, i.e.\ the use of the approximation at the endpoint of the last subinterval for the start point of the next one, could cause a problem in this context since the corresponding operator is independent of $\Delta t$ and would thus add a constant term to our estimations.







\bibliographystyle{gtart}      

\bibliography{refs}   

\end{document}